\documentclass[runningheads]{llncs}

\usepackage{amsmath, amsfonts, amssymb, graphicx, xcolor, dsfont, enumerate, dcolumn,subcaption}
\usepackage[numbers,square, comma, compress]{natbib}
\usepackage{algorithm}
\usepackage{algorithmic}
\usepackage{makecell}
\usepackage{threeparttable}

\let\doendproof\endproof
\renewcommand\endproof{\hfill\qed\doendproof}

\usepackage{hyperref}
\usepackage{xurl}
\usepackage{bm}
\usepackage[capitalise]{cleveref}
\PassOptionsToPackage{hyphens}
{url}\usepackage{hyperref}

\usepackage{tikz}

\newtheorem{observation}{Observation}

\newtheorem{prop}{Proposition}

\begin{document}

\title{On vertex-girth-regular graphs: (Non-)existence, bounds and enumeration}

\titlerunning{On vertex-girth-regular graphs}
%

\author{Robert Jajcay\inst{1}, Jorik Jooken\inst{2}, \and
István Porupsánszki\inst{3} }

\authorrunning{Jajcay et al.}

\institute{Department of Algebra and Geometry, Faculty of Mathematics, Physics and Informatics, Mlynsk\' a dolina, 842 48 Bratislava, Slovakia\\
\email{robert.jajcay@fmph.uniba.sk}
\and
Department of Computer Science, KU Leuven Kulak, 8500 Kortrijk, Belgium\\
\email{jorik.jooken@kuleuven.be}
\and
{Department of Geometry and MTA-ELTE Geometric and Algebraic Combinatorics Research Group, Eötvös
Loránd University, 1117 Budapest, Pázmány s. 1/c, Hungary}\\
\email{rupsansz@gmail.com}\\
}

\date{}

\maketitle

\begin{abstract}
A vertex-girth-regular $vgr(v,k,g,\lambda)$-graph is a $k$-regular graph of girth $g$ and order $v$ in which every vertex belongs to exactly $\lambda$ 
cycles of length $g$. While all vertex-transitive graphs are necessarily vertex-girth-regular, the majority of vertex-girth-regular graphs are not
vertex-transitive. Similarly, while many of the smallest $k$-regular graphs of girth $g$, the so-called $(k,g)$-cages, are vertex-girth-regular, infinitely
many vertex-girth-regular graphs of degree $k$ and girth $g$ exist for many pairs $k,g$. Due to these connections, the study of 
vertex-girth-regular graphs promises insights into the relations between the classes of extremal, highly symmetric, and locally regular graphs of given
degree and girth. This paper lays the foundation to such study by investigating the fundamental properties of $vgr(v,k,g,\lambda)$-graphs, specifically the relations necessarily satisfied by the parameters $v,k,g$ and $\lambda$ to admit the existence of a corresponding vertex-girth-regular graph, by presenting constructions of infinite families of $vgr(v,k,g,\lambda)$-graphs, and by establishing lower bounds on the number $v$ of vertices in a $vgr(v,k,g,\lambda)$-graph. It also includes computational results determining the orders of smallest cubic and quartic graphs of small girths.

%
\end{abstract}


\section{Introduction}
\label{sec:introduction}
The motivation for considering vertex-girth-regular graphs comes from two seemingly disconnected areas. 

The first is the Cage Problem; a part of Extremal Graph Theory where one searches for $k$-regular graphs of girth $g$,
$ (k,g) $-graphs, of the smallest possible order; called $ (k,g) ${\em -cages}~\cite{EJ08}. Even though this problem has a large number of practical applications (e.g., in Network Design) and has been studied since the 1940's, very few orders of $(k,g)$-cages are known. This is due, among other reasons, to our lack of understanding of the structure of cages as well as the considerable size of the search spaces associated with searching for smallest graphs of given parameters $k,g$. To fill the existing void, it is often useful to rely on heuristics - observations based on the structure of known cages or smallest known $(k,g)$-graphs. One such insight is that small $(k,g)$-graphs often exhibit a high level of regularity and tend to look similar with respect to each of their vertices. Transforming this observation into a structural description suggests for example that each vertex of a small $(k,g)$-graph is contained in a similar number of \emph{girth-cycles}; cycles of length $g$. This ultimately leads to the study of $(k,g)$ \emph{vertex-girth-regular graphs} which are $k$-regular graphs of girth $g$ in which every vertex is contained in the same number of girth-cycles; which we shall usually denote by $\lambda$. Since our partial aim is to shed light on structural properties of small $(k,g)$-graphs, we also address the question of smallest orders of vertex-girth-regular graphs for a given triple $k,g$ and $\lambda$. Similar questions concerning girth-regularity of regular graphs have been studied in a series of papers focusing on the orders of smallest $k$-regular graphs of girth $g$ in which every \emph{edge} is contained in the same number of girth-cycles, called \emph{edge-girth-regular graphs} $egr(n,k,g,\lambda)$ \cite{DFJR21,JKM18}, and in the paper \cite{PV19} where the authors consider \emph{girth-regular graphs} $gr(n,k,g,\mathbf{a})$: $k$-regular graphs of girth $g$ having the property that the \emph{signature} of every vertex is the same, where the signature $\mathbf{a} = \{a_1, a_2, \ldots, a_k\}$ of a vertex $u$ represents the number of times the $k$ edges adjacent to $u$ are contained in girth-cycles. Clearly, both of the above classes of graphs are also vertex-girth-regular. However, in our definition, we do not make any assumptions about the distribution of girth-cycles among the edges adjacent to a vertex. It is also worth noting that the number of girth-cycles through any vertex in an edge-girth-regular graph is necessarily a multiple of half of the degree, $\frac{k}{2}$.

The second source of inspiration for our study of vertex-girth-regular graphs is the class of vertex-transitive graphs which are necessarily vertex-girth-regular (and much more; since they have the property that the number of cycles of any specific length through each vertex is the same). Thus, in some sense, the study of vertex-girth-regular graphs is the study of the connection between vertex-transitivity and girth-regularity much the same way as the study of edge-girth-regular graphs is connected to edge-transitivity. It is important to note that neither edge- nor vertex-girth-regularity imply edge- or vertex-transitivity; as amply exhibited by our example of an edge-girth-regular tetravalent graph on $20$ vertices which has a trivial automorphism group (see Fig.~\ref{fig:trivialAutomorphism}).

\begin{figure}[h!]
\begin{center}
\begin{tikzpicture}[scale=1.6]
  \def\inter{1.5}
  \def\vertDist{0.5}
  \def\k{4}

  \foreach \i in {1}
    {
        \fill (\inter * \i, 0*\vertDist) circle (1.2pt);
    }
    \foreach \i in {2}
    {
        \fill (\inter * \i, -0.4*\vertDist) circle (1.2pt);
    }
    \foreach \i in {3,...,\k}
    {
        \fill (\inter * \i, 0*\vertDist) circle (1.2pt);
    }

    
    \foreach \i in {0,...,\k}
    {
        \fill ({\inter * (\i+0.5)}, 1*\vertDist) circle (1.2pt);
    }

    
    \foreach \i in {1}
    {
        \fill (\inter * \i, 2*\vertDist) circle (1.2pt);
    }
    \foreach \i in {2}
    {
        \fill (\inter * \i, 1.3*\vertDist) circle (1.2pt);
    }
    \foreach \i in {3,...,\k}
    {
        \fill (\inter * \i, 2*\vertDist) circle (1.2pt);
    }

    \fill ({\inter * (1.8)}, 3*\vertDist) circle (1.2pt);
    
    \foreach \i in {2,...,\k}
    {
        \fill ({\inter * (\i+0.5)}, 3*\vertDist) circle (1.2pt);
    }

    
    \foreach \i in {2}
    {
        \fill ({\inter * (\i+0.5)}, 4*\vertDist) circle (1.2pt);
    }
    
    
    \foreach \i in {0,\k}
    {
        \fill ({\inter * (\i+0.5)}, 5*\vertDist) circle (1.2pt);
    }


    
    \foreach \i in {1}
    {
        \draw (\inter * \i, 0*\vertDist) -- ({\inter * (\i-0.5)}, 1*\vertDist);
    }
    \foreach \i in {2}
    {
        \draw (\inter * \i, -0.4*\vertDist) -- ({\inter * (\i-0.5)}, 1*\vertDist);
    }
    \foreach \i in {3,...,\k}
    {
        \draw (\inter * \i, 0*\vertDist) -- ({\inter * (\i-0.5)}, 1*\vertDist);
    }

    \foreach \i in {1}
    {
        \draw (\inter * \i, 0*\vertDist) -- ({\inter * (\i+0.5)}, 1*\vertDist);
    }
    \foreach \i in {2}
    {
        \draw (\inter * \i, -0.4*\vertDist) -- ({\inter * (\i+0.5)}, 1*\vertDist);
    }
    \foreach \i in {3,...,\k}
    {
        \draw (\inter * \i, 0*\vertDist) -- ({\inter * (\i+0.5)}, 1*\vertDist);
    }


    \foreach \i in {1}
    {
        \draw (\inter * \i, 2*\vertDist) -- ({\inter * (\i-0.5)}, 1*\vertDist);
    }
    \foreach \i in {2}
    {
        \draw (\inter * \i, 1.3*\vertDist) -- ({\inter * (\i-0.5)}, 1*\vertDist);
    }
    \foreach \i in {3,...,\k}
    {
        \draw (\inter * \i, 2*\vertDist) -- ({\inter * (\i-0.5)}, 1*\vertDist);
    }

        \foreach \i in {1}
    {
        \draw (\inter * \i, 2*\vertDist) -- ({\inter * (\i+0.5)}, 1*\vertDist);
    }
    \foreach \i in {2}
    {
        \draw (\inter * \i, 1.3*\vertDist) -- ({\inter * (\i+0.5)}, 1*\vertDist);
    }
    \foreach \i in {3,...,\k}
    {
        \draw (\inter * \i, 2*\vertDist) -- ({\inter * (\i+0.5)}, 1*\vertDist);
    }

    \draw (\inter * 4, 2*\vertDist) -- ({\inter * (4+0.5)}, 3*\vertDist);
    \draw (\inter * 1, 0*\vertDist) -- ({\inter * (4+0.5)}, 3*\vertDist);
    \draw ({\inter * (4+0.5)}, 5*\vertDist) -- ({\inter * (4+0.5)}, 3*\vertDist);
    \draw ({\inter * (2+0.5)}, 4*\vertDist) -- ({\inter * (4+0.5)}, 3*\vertDist);

    \draw ({\inter * (2+0.5)}, 4*\vertDist) -- ({\inter * (3+0.5)}, 3*\vertDist);
    \draw ({\inter * (2+0.5)}, 4*\vertDist) -- ({\inter * (1.8)}, 3*\vertDist);
    \draw ({\inter * (2+0.5)}, 4*\vertDist) -- (\inter * 4, 0*\vertDist);

    \draw (\inter * 3, 2*\vertDist) -- ({\inter * (2+0.5)}, 3*\vertDist);
    \draw ({\inter * (1.8)}, 3*\vertDist) -- ({\inter * (2+0.5)}, 3*\vertDist);
    \draw ({\inter * (4+0.5)}, 1*\vertDist) -- ({\inter * (2+0.5)}, 3*\vertDist);
    \draw ({\inter * (4+0.5)}, 5*\vertDist) -- ({\inter * (2+0.5)}, 3*\vertDist);

    \draw ({\inter * (0+0.5)}, 5*\vertDist) -- ({\inter * (1.8)}, 3*\vertDist);
    \draw ({\inter * (0+0.5)}, 5*\vertDist) -- ({\inter * (0+0.5)}, 1*\vertDist);
    \draw ({\inter * (0+0.5)}, 5*\vertDist) -- (\inter * 2, 1.3*\vertDist);
    \draw ({\inter * (0+0.5)}, 5*\vertDist) -- ({\inter * (4+0.5)}, 5*\vertDist);

    \draw (\inter * 1, 2*\vertDist) -- (\inter * 3, 2*\vertDist);
    \draw (\inter * 1, 2*\vertDist) -- ({\inter * (4+0.5)}, 1*\vertDist);

    \draw ({\inter * (0+0.5)}, 1*\vertDist) -- ({\inter * (3+0.5)}, 3*\vertDist);
    \draw (\inter * 4, 2*\vertDist) -- ({\inter * (3+0.5)}, 3*\vertDist);
    \draw (\inter * 2, 1.3*\vertDist) -- ({\inter * (3+0.5)}, 3*\vertDist);

    \draw (\inter * 2, -0.4*\vertDist) -- ({\inter * (1.8)}, 3*\vertDist);

    \draw (\inter * 2, -0.4*\vertDist) -- (\inter * 4, 0*\vertDist);

    \draw (\inter * 1, 0*\vertDist) -- (\inter * 3, 0*\vertDist);
    
    \draw ({\inter * (4+0.5)}, 5*\vertDist) -- (\inter * 3, 0*\vertDist);

\end{tikzpicture}
\end{center}
\caption{An $egr(20,4,4,1)$-graph which is asymmetric (it only has the trivial automorphism).}
\label{fig:trivialAutomorphism}
\end{figure}
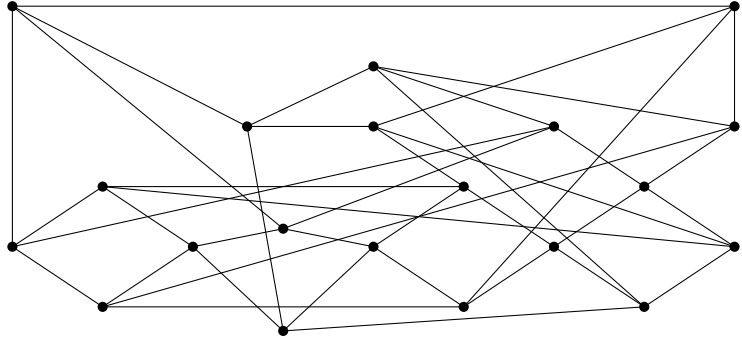

\section{Preliminaries and notation}
We first recall the Moore bound $M(k,g)$, which provides a lower bound for the order of a $(k,g)$-graph:


\begin{observation}[Folklore]
    Let $G$ be a $(k,g)$-graph of order $n$. Then, we have:
\begin{align*}
    n \geq M(k,g) = \begin{cases}
        1+\sum_{i=0}^{(g-3)/2} k(k-1)^i& \textrm{if } g\textrm{ is odd},\\
        2\sum_{i=0}^{(g-2)/2} (k-1)^i& \textrm{if } g\textrm{ is even}.\\
    \end{cases}
\end{align*}
\end{observation}
If a $(k,g)$-graph attains this bound, we call it a \emph{Moore graph}.

As we already stated in the Introduction, for integers $v$, $k$, $g$ and $\lambda$, a vertex-girth-regular $vgr(v,k,g,\lambda)$-graph is a $k$-regular graph with girth $g$ on $v$ vertices such that every vertex is contained in exactly $\lambda$ cycles of length $g$. 

It is important to note that not every triple $(k,g,\lambda)$ admits the existence of a $vgr(v,k,g,\lambda)$-graph. For example, we can derive the following obvious upper bounds on $\lambda$. Several of the proofs in this paper will follow the same setup as the proof of the following proposition:

\begin{prop}\label{prop:elementalProperties}
    Let $G$ be a $vgr(v,k,g,\lambda)$-graph. Then the following hold:
    \begin{enumerate}[(i)]
      \item if $g$ is odd, then $\lambda \leq \frac{k(k-1)^{\frac{g-1}{2}}}{2}$; with equality if and only if $G$ is a Moore graph;
    \item if $g$ is even, then $\lambda \leq \frac{k(k-1)^{\frac{g}{2}}}{2}$; with equality if and only if $G$ is a Moore graph.
    \end{enumerate}
\end{prop}
\begin{proof}
(i): Consider any vertex $u_1 \in V(G)$ and let $\mathcal{T}^{u_1}_{k,\frac{g-1}{2}}$ be a subgraph of $G$ such that $V(\mathcal{T}^{u_1}_{k,\frac{g-1}{2}})$ consists of all vertices which are at distance at most $\frac{g-1}{2}$ from $u_1$ and $E(\mathcal{T}^{u_1}_{k,\frac{g-1}{2}})$ contains all edges of $E(G)$ between vertices of $V(\mathcal{T}^{u_1}_{k,\frac{g-1}{2}})$ except for edges between vertices at distance $\frac{g-1}{2}$ from $u_1$. Since $G$ has girth $g$, $\mathcal{T}^{u_1}_{k,\frac{g-1}{2}}$ is a tree (called a Moore tree) and $|V(\mathcal{T}^{u_1}_{k,\frac{g-1}{2}})|=M(k,g)$. Now there is a one-to-one correspondence between girth-cycles of $G$ containing $u_1$ and edges that have two endpoints that are leaves of $\mathcal{T}^{u_1}_{k,\frac{g-1}{2}}$. Since each vertex of $G$ has degree $k$ and $\mathcal{T}^{u_1}_{k,\frac{g-1}{2}}$ has $k(k-1)^{\frac{g-3}{2}}$ leaves, we obtain that the number of edges between two leaves of $\mathcal{T}^{u_1}_{k,\frac{g-1}{2}}$ is bounded from above by $\frac{k(k-1)^{\frac{g-1}{2}}}{2}$ and this bound is attained if and only if $G$ is a Moore graph.\\
(ii): The even girth case uses a similar setup as the odd girth case, but requires different arguments. Consider any edge $u_1u_2 \in E(G)$ and let $\mathcal{T}^{u_1,u_2}_{k,\frac{g}{2}-1}$ be the subtree of $G$ consisting of the edge $u_1u_2$ and  two disjoint trees rooted at respectively $u_1$ and $u_2$ such that the leaves of the two trees are at distance $\frac{g}{2}-1$ from $u_1$ and $u_2$, respectively. Note that $|V(\mathcal{T}^{u_1,u_2}_{k,\frac{g}{2}-1})|=M(k,g)$. Let $D_{u_1}$ and $D_{u_2}$ be the set of leaves of $\mathcal{T}^{u_1,u_2}_{k,\frac{g}{2}-1}$ at distance $\frac{g}{2}-1$ from $u_1$ and $u_2$, respectively. Let $v_1, v_2, \ldots, v_{k-1}$ be the $k-1$ neighbors of $u_1$ different from $u_2$ and let $D_{u_1,1}, D_{u_1,2}, \ldots, D_{u_1,k-1}$ be the set of leaves at distance $\frac{g}{2}-2$ from $v_1, v_2, \ldots, v_{k-1}$, respectively. The set of girth-cycles containing $u_1$ can be partitioned in three sets $\mathcal{A}_{u_1}, \mathcal{B}_{u_1}$ and $\mathcal{C}_{u_1}$, where $\mathcal{A}_{u_1}$ is the set of girth-cycles that contain exactly one edge with an endpoint in $D_{u_1}$ and one endpoint in $D_{u_2}$, $\mathcal{B}_{u_1}$ is the set of girth-cycles that contain exactly two edges with one of their endpoints in $D_{u_1}$ and a shared endpoint in $D_{u_2}$, and $\mathcal{C}_{u_1}$ is the set of girth-cycles that contain exactly two edges with one endpoint in $D_{u_1}$ and one endpoint in $V(G) \setminus V(\mathcal{T}^{u_1,u_2}_{k,\frac{g}{2}-1})$. In a similar fashion, we partition the set of girth-cycles containing $u_2$ in three sets $\mathcal{A}_{u_2}$, $\mathcal{B}_{u_2}$ and $\mathcal{C}_{u_2}$. Note that each girth-cycle in $\mathcal{A}_{u_1}$ contains the edge $u_1u_2$, whereas none of the girth-cycles in $\mathcal{B}_{u_1}$ and $\mathcal{C}_{u_1}$ contain $u_1u_2$. Since every vertex in $G$ has degree $k$ and $|D_{u_1}|=(k-1)^{\frac{g-2}{2}}$, we obtain $|\mathcal{A}_{u_1}| \leq (k-1)^{\frac{g}{2}}$. Since $G$ has girth $g$, there does not exist a vertex $v \in (V(G) \setminus V(\mathcal{T}^{u_1,u_2}_{k,\frac{g}{2}-1})) \cup D_{u_2}$ and an integer $1 \leq i \leq k-1$ such that $v$ is adjacent to two distinct vertices in $D_{u_1,i}$. Therefore, every vertex $v \in (V(G) \setminus V(\mathcal{T}^{u_1,u_2}_{k,\frac{g}{2}-1})) \cup D_{u_2}$ is adjacent to at most $k-1$ vertices in $D_{u_1}$. Since each vertex in $G$ has degree $k$, we obtain $|\mathcal{B}_{u_1}| + |\mathcal{C}_{u_1}| \leq (k-1)^{\frac{g-2}{2}} \frac{(k-1)(k-2)}{2}$. This yields $\lambda = |\mathcal{A}_{u_1}|+|\mathcal{B}_{u_1}|+|\mathcal{C}_{u_1}| \leq (k-1)^{\frac{g}{2}}+(k-1)^{\frac{g-2}{2}} \frac{(k-1)(k-2)}{2}=\frac{k(k-1)^{\frac{g}{2}}}{2}$ and if equality occurs, $|\mathcal{A}_{u_1}|=(k-1)^{\frac{g}{2}}$ and therefore $|\mathcal{B}_{u_1}|=(k-1)^{\frac{g-2}{2}} \frac{(k-1)(k-2)}{2}$ and $|\mathcal{C}_{u_1}|=0$, which implies that $G$ is a Moore graph. Conversely, if $G$ is a Moore graph we have $|\mathcal{A}_{u_1}|=(k-1)^{\frac{g}{2}}$, $|\mathcal{B}_{u_1}|=(k-1)^{\frac{g-2}{2}} \frac{(k-1)(k-2)}{2}$, $|\mathcal{C}_{u_1}|=0$ and  $\lambda = \frac{k(k-1)^{\frac{g}{2}}}{2}$. 
\end{proof}

Beside the upper bounds on $\lambda$ derived from the properties of cages stated in Proposition~\ref{prop:elementalProperties}, vertex-girth-regular graphs also do not exist in cases when $\lambda$ is close to the upper bounds stated in there. 

For example, when considering cubic vertex-girth-regular graphs of girth $3$, it is easy to see that $K_4$ is a $vgr(4,3,3,3)$-graph. Moreover, 
it is the unique $3$-regular graph of girth $3$ and $\lambda = 3 $, which is the maximal $\lambda $ in any $ vgr(n,3,3,\lambda)$-graph. 
Also, it is not very hard to construct a $3$-regular graph of girth $3$ and $\lambda = 1 $, as shown in Fig.~\ref{fig:(6,3,3,1)}.

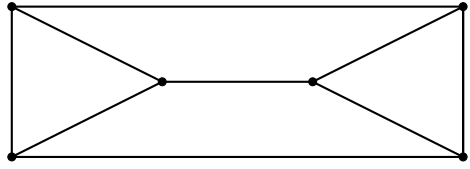
\begin{figure}
\begin{center}
\begin{tikzpicture}
\draw[fill=black] (0,0) circle (1.5pt);
\draw[fill=black] (6,0) circle (1.5pt);
\draw[fill=black] (2,1) circle (1.5pt);
\draw[fill=black] (4,1) circle (1.5pt);
\draw[fill=black] (0,2) circle (1.5pt);
\draw[fill=black] (6,2) circle (1.5pt);
\draw[thick] (0,0) -- (6,0) -- (6,2) -- (0,2) -- (0,0) -- (2,1) -- (0,2);
\draw[thick] (2,1) -- (4,1) -- (6,2) -- (6,0) -- (4,1);
\end{tikzpicture}
\end{center}
\caption{A $vgr(6,3,3,1)$-graph.}
\label{fig:(6,3,3,1)}
\end{figure}

However, there is no cubic vertex-girth-regular graph of girth $3$ in which every vertex belongs to $2$ girth-cycles. This is a consequence of the following lemma.

\begin{lemma}\label{easy}
Let $ k \geq 3 $.
There is no integer $v$ such that a $vgr(v,k,3,{k \choose 2}-1)$-graph exists. \end{lemma}

\begin{proof}
It is easy to see that the complete graph $K_{k+1}$ is a $vgr(k+1,k,3,{k \choose 2})$-graph, and that no $\lambda $ in a $ vgr(v,k,3,\lambda)$-graph is greater than $ {k \choose 2}$. Consider a potential $vgr(v,k,3,{k \choose 2}-1)$-graph $G$. 
Let $u$ be any vertex in $G$, and let $v_1,v_2, \ldots, v_k$ be the 
neighbors of $u$. In similarity to $K_{k+1}$, any pair of distinct neighbors of $u$ but exactly one must be adjacent (recall that $u$ is assumed to be contained in ${k \choose 2}-1$ triangles). Without loss of generality, we may assume that $v_1$ is the vertex that is not connected to all the other neighbors of $u$. As $v_1$ is assumed to be of degree $k$, it must have a neighbor $w$ different from the vertices $u,v_1,v_2,\ldots,v_k$.
Further note that any triangle containing $v_1$ must contain two adjacent neighbors of $v_1$, while $w$ is not adjacent to any of the neighbors of $v_1$
among the vertices $u,v_2,\ldots,v_k$ as they are already assumed to be of degree $k$. This means that $v_1$ is contained in at most $ {k-1 \choose 2} $ triangles. Since $ k \geq 3 $,
${k-1 \choose 2} < {k \choose 2} - 1 $. Therefore $G$ is not a $vgr(v,k,3,{k \choose 2}-1)$-graph; which completes the argument.
\end{proof}

Thus, unlike the case of $(k,g)$-graphs which exist for any pair of parameters
$ k,g \geq 3 $ \cite{S63}, the question of the existence of at least
one $vgr(v,k,g,\lambda)$-graph for a given triple $(k,g,\lambda)$ necessarily precedes the question of the order of a smallest such graph. That is why we begin our paper with Section~\ref{sec:Existence} in which we show the 
existence of $vgr(v,k,g,\lambda)$-graphs for large classes of triples $(k,g,\lambda)$. This is followed by Section~\ref{sec:LBounds} where we derive several natural lower bounds on the orders of $vgr(v,k,g,\lambda)$-graphs.
In analogy to the Cage Problem, we define $n(k,g,\lambda)$ to be {\em the smallest integer $v$ such that a $vgr(v,k,g,\lambda)$-graph exists} (or $\infty$ otherwise), and 
similarly, $n_2(k,g,\lambda)$ to be {\em the smallest integer $v$ such that a {\em bipartite} $vgr(v,k,g,\lambda)$ exists} (and $\infty$ otherwise). 
We then present further non-existence results in Section~\ref{sec:nonExistence}, and conclude the paper with a number of computational results in which we determine the orders of smallest vertex-girth-regular graphs for various sets of small parameter triples $(k,g,\lambda)$.

Before exiting this section, let us revisit Lemma~\ref{easy}. As shown by the existence of a $vgr(6,3,3,1)$-graph, at least in case of $k=3$, the lemma cannot be strengthened. Even though we were unable to find a general proof for the claim that no $ vgr(v,k,3,{k \choose 2} - \epsilon) $-graphs exist for $ 0 < \epsilon < \frac{k-1}{2} $, the result appears feasible
(especially in view of the analogous Theorem~\ref{non-exist-odd}). 
However, as shown by the existence of a $vgr(6,4,3,4)$-graph, the upper bound $ \epsilon < \frac{k-1}{2} $ is sharp and cannot be replaced by $ \epsilon \leq \frac{k-1}{2} $ 
appearing in Theorem~\ref{non-exist-odd}. 
Finally note that $ vgr(2k,k,3,{k \choose 2} - (k-1)) $-graphs 
exist for all $k \geq 3$. They can be constructed in the same way as the
$vgr(6,3,3,1)$-graph depicted above by joining two copies of the complete graph $K_k$ via a perfect matching connecting each vertex in one of the copies to exactly one vertex in the other.

\section{Existence results} \label{sec:Existence}

In contrast to Lemma~\ref{easy}, in this section we will show that $vgr(v,k,g,\lambda)$-graphs exist in many cases. We start by showing, based on the idea of generalized truncation, that one can construct vertex-girth-regular graphs of increasing degree starting from an arbitrary $vgr(v,k,g,\lambda)$-graph.

\begin{prop}\label{prop:generalizedTruncation}
If a $vgr(v,k,g,\lambda)$-graph exists, then there exist infinitely many integers $v'$ such that a $vgr(v',k+1,g,\lambda)$-graph also exists.
\end{prop}
\begin{proof}
    Let $G$ be a $vgr(v,k,g,\lambda)$-graph. Consider any $v$-regular graph $G'$ of girth strictly larger than $g/2$ (such a graph exists for any choice of $v$ and $g$ as shown in~\cite{S63}). Construct the graph $H$ by generalized truncation: $H$ is obtained by replacing every vertex in $G'$ by a copy of the graph $G$ (if $u$ is a vertex of $G'$ with neighbors $w_1, w_2, \ldots, w_v$ and $x_1, x_2, \ldots, x_v$ are the vertices of $G$ in an arbitrary order, then the vertex $u$ is replaced by the graph $G$ and the edges $w_1x_1, w_2x_2, \ldots, w_vx_v$ are added). Now each cycle in $H$ corresponds either to a cycle in $G$ or to a cycle in $G'$ in which each vertex is replaced by a path consisting of at least two vertices. Therefore $H$ is a $(k+1)$-regular graph with girth $g$ such that each vertex is contained in $\lambda$ girth-cycles.
\end{proof}

Since every girth-cycle of a graph containing a vertex $u$ contains exactly two edges incident with $u$, every edge-girth-regular graph is also vertex-girth-regular.
\begin{observation}
\label{obs:EgrToVgr}
    If $G$ is an $egr(v,k,g,\lambda)$-graph, then $G$ is a $vgr(v,k,g,\frac{k\lambda}{2})$-graph.
\end{observation}

We now recall three theorems from \cite{JKM18} concerning existence of edge-girth-regular graphs that we will use later:
\begin{theorem}[Th. 3.4 in~\cite{JKM18}]
\label{th:egrGirthAtLeast6}
For every $k \geq 3$ and every $g \geq 6$, there exist infinitely many $egr(v, k, g, 2)$-graphs.
\end{theorem}
\begin{theorem}[Th. 4.1 in~\cite{JKM18}]
\label{th:egrGirth3}
    The Cartesian product of an $egr(v_1, k_1, 3, \lambda)$-graph and an $egr(v_2, k_2, 3, \lambda)$-graph is an $egr(v_1v_2, k_1 + k_2, 3, \lambda)$-graph.
\end{theorem}
\begin{theorem}[Th. 4.3 in~\cite{JKM18}]
\label{th:powerOfR}
For every $r \geq 2$ and $g \geq 3$, there exist infinitely many $egr(v, 2^r, g, 1)$-graphs.
\end{theorem}

We are now ready to prove the main existence theorem of this section which shows that vertex-girth-regular graphs exist for many $k, g$ and $\lambda$:
\begin{theorem}
There are infinitely many integers $v$ such that a $vgr(v,k,g,\lambda)$-graph exists:
\begin{enumerate}[(i)]
    \item for $\lambda=1$ and all integers $k,g \geq 3$;
    
    \item for $\lambda=2$ and all integers $k \geq 4$, $g \geq 3$;
    
    \item for all integers $\lambda \geq 3$, $k \geq \lambda$, $g \geq 3$, $g \notin \{4,5\}$.
\end{enumerate}
\end{theorem}
\begin{proof}
(i) By recursively applying Proposition~\ref{prop:generalizedTruncation} to cycles (i.e., $vgr(g,2,g,1)$-graphs), we obtain the existence of infinitely many $vgr(v,k,g,1)$-graph for all integers $k, g \geq 3$.\\
(ii) By applying Theorem~\ref{th:powerOfR} for $r=2$, we obtain the existence of infinitely many $egr(v,4,g,1)$-graphs for all $g \geq 3$ and thus infinitely many $vgr(v,4,g,2)$-graphs by applying Observation~\ref{obs:EgrToVgr}. Finally, recursively applying Proposition~\ref{prop:generalizedTruncation} to these graphs proves (ii).\\
(iii) We first deal with the case $g=3$. As discussed in~\cite{GJ24}, there exist an $egr(v_1,3,3,2)$-graph $G_1$ on 4 vertices, an $egr(v_2,4,3,2)$-graph $G_2$ on 6 vertices and an $egr(v_3,5,3,2)$-graph $G_3$ on 12 vertices. For all integers $k \geq 6$, one can construct the graph $G_{k-2}$ by taking the Cartesian product of $G_1$ and $G_{k-5}$. Because of Theorem~\ref{th:egrGirth3}, these graphs are all $k$-regular graphs with girth 3 such that each edge is contained in exactly 2 triangles. Hence, this yields the existence of a $vgr(v,k,3,k)$-graph for each $k \geq 3$ because of Observation~\ref{obs:EgrToVgr}. By recursively applying Proposition~\ref{prop:generalizedTruncation} to these graphs, we obtain infinitely many $vgr(v,k,3,\lambda)$-graphs and we are done with the girth 3 case.

We can combine Theorem~\ref{th:egrGirthAtLeast6} and Observation~\ref{obs:EgrToVgr} to obtain the existence of a $vgr(v,k,g,k)$-graph for each $g \geq 6$ and $k \geq 3$. Finally, by recursively applying Proposition~\ref{prop:generalizedTruncation} to these graphs, we obtain for each fixed $(k,g,\lambda)$ (where $\lambda \geq 3$, $k \geq \lambda$ and $g \geq 6$) the existence of infinitely many $vgr(v,k,g,\lambda)$-graphs.
\end{proof}

\section{Lower bounds on $n(k,g,\lambda)$} \label{sec:LBounds}

In this section, we present lower bounds on the order of extremal vertex-girth-regular graphs (i.e., vertex-girth-regular graphs of the smallest order). We get the lower bounds by generalizing already existing bounds for edge-girth-regular and girth-regular graphs \cite{APKP23,DFJR21,P23}. We will use the fact that every vertex of a $vgr(n,k,g,\lambda)$-graph has an edge that is contained in at least/at most $\frac{2\lambda}{k}$ distinct girth-cycles. In order to prove this fact, we choose an arbitrary vertex $u$ and denote its signature by $\mathbf{a}=\{a_1,\ldots, a_k\}$, where $a_1 \geq a_2 \geq ... \geq a_k$. The sum of these numbers is two times the number of girth-cycles through $u$, which is equal to $\lambda$ by definition. Hence, the average of the signature is $\frac{2\lambda}{k}$. Therefore
\begin{equation} a_k\le \left\lfloor\frac{2\lambda}{k}\right\rfloor\le\left\lceil\frac{2\lambda}{k}\right\rceil\le a_1.\end{equation}
First, we present a generalization of the combinatorial lower bounds.
\begin{theorem}
    Let $G$ be a $vgr(n,k,g,\lambda)$-graph, where $g=2h$ is an even number. Then
    $$n\ge 2\frac{(k-1)^h-1}{k-2}+\left\lceil \frac{2(k-1)^h-2\left\lfloor\frac{2\lambda}{k}\right\rfloor}{k}\right\rceil.$$
    Moreover, if $G$ is bipartite then
    $$n\ge 2\frac{(k-1)^h-1}{k-2}+2\left\lceil \frac{(k-1)^h-\left\lfloor\frac{2\lambda}{k}\right\rfloor}{k}\right\rceil.$$
\end{theorem}
\begin{proof}
There is an edge that is contained in $\Lambda \le\left\lfloor\frac{2\lambda}{k}\right\rfloor$ distinct girth-cycles. For that particular edge, we follow the proof of Theorem 2.3 and Theorem 5.1 in \cite{DFJR21}, and we immediately get the lower bound: we consider the set of vertices that are at a distance $h$ from the chosen edge. The number of edges with both endpoints in the set of these vertices equals the number of girth-cycles through the chosen edge, which is precisely $\Lambda$. Hence, the number of edges that leave the Moore tree is $2(k-1)^h-2\Lambda$. The graph is $k$-regular so that we can give a lower bound on the number of vertices outside the Moore tree:
\begin{align*}
    n&\ge 2\frac{(k-1)^h-2}{k-2}+\left\lceil\frac{2(k-1)^h-2\Lambda}{k}\right\rceil\\
    &\ge 2\frac{(k-1)^h-2}{k-2}+\left\lceil \frac{2(k-1)^h-2\left\lfloor\frac{2\lambda}{k}\right\rfloor}{k}\right\rceil.
\end{align*}
In the bipartite case, there is a slight improvement on the lower bound. There are $2(k-1)^h-2\Lambda$ edges that leave the Moore tree. Half of them have an endpoint in one part of the graph, and the rest of them have an endpoint in the other part. Hence
\begin{align*}
    n&\ge 2\frac{(k-1)^h-2}{k-2}+2\left\lceil\frac{(k-1)^h-\Lambda}{k}\right\rceil\\
    &\ge 2\frac{(k-1)^h-2}{k-2}+2\left\lceil \frac{(k-1)^h-\left\lfloor\frac{2\lambda}{k}\right\rfloor}{k}\right\rceil.
\end{align*}
\end{proof}
\begin{theorem}
    Let $G$ be a $vgr(n,k,g,\lambda)$-graph, where $g=2h+1$ is an odd number. Then
    $$n\ge n(k,g,\lambda)\ge \frac{k(k-1)^h-2}{k-2}+\left\lceil\frac{k(k-1)^h-2\lambda}{k}\right\rceil.$$
\end{theorem}
\begin{proof}
We can apply the same argument as in Theorem 2.3 in \cite{DFJR21}: we choose an arbitrary vertex $v$ and consider the set of vertices that are at a distance $h$ from $v$. Then, the number of girth-cycles through $v$ equals the number of edges with both endpoints in this set. It is exactly $\lambda$. Hence, we can give a lower bound on the number of vertices outside the Moore tree. So, the order of the graph is at least
$$n\ge n(k,g,\lambda)\ge \frac{k(k-1)^h-2}{k-2}+\left\lceil\frac{k(k-1)^h-2\lambda}{k}\right\rceil.$$
\end{proof}
The next lower bound for even girth is a combinatorial one that gives a lower bound on the number of vertices outside the Moore tree. It is a straightforward generalization of Theorem $4.4$ in \cite{APKP23}. The proof is based on counting the number of girth-cycles through an arbitrary vertex, avoiding one of its edges.
\begin{theorem}
Let $G$ be a $vgr(n,k,g,\lambda)$-graph, where $g=2h$ is an even number. Suppose that there exists an edge that is contained in $\Lambda$ distinct girth-cycles. Then
$$
    \begin{aligned}
        n\ge2\frac{(k-1)^h-1}{k-2} +\left\lceil\frac{\left((k-1)^h-\Lambda\right)^2}{2\lambda-3\Lambda+(k-1)^h-2\max\left(0,\left\lceil\frac{\Lambda^2}{2(k-1)^{(h-1)}}-\frac{\Lambda}{2}\right\rceil\right)}\right\rceil.
    \end{aligned}
$$
\end{theorem}
\begin{proof}
We follow the same argument as in the proof of Theorem $4.4$ in \cite{APKP23}.
Choose an arbitrary edge $uv$ that is contained in exactly $\Lambda$ distinct $g$-cycles. We define the set $D_u$ of vertices as follows: $w\in D_u$ if and only if the length of the shortest $uw$-path is $h-1$, and the length of the shortest $vw$-path is $h$. Similarly,  $w\in D_v$ if and only if the length of the shortest $vw$-path is $h-1$, and the length of the shortest $uw$-path is $h$. The number of edges between $D_u$ and $D_v$ is exactly $\Lambda$. Hence, the number of girth-cycles through $u$ that do not contain the edge $uv$ is $\lambda-\Lambda$. There are two types of these cycles: the ones that have a vertex in $D_v$ and the ones that have a vertex at a distance $h$ from $u$ and outside of $D_v$. Denote the latter set of these vertices with $M$ (see Fig.~\ref{fig:evenGirthMooreTree}).\\
\begin{figure}
\begin{center}
    \includegraphics[scale=0.8]{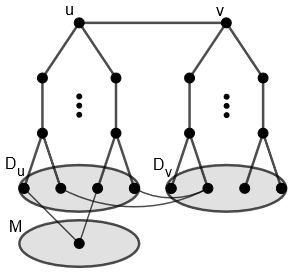}
\end{center}
\caption{The sets $D_u$, $D_v$ and $M$.}
\label{fig:evenGirthMooreTree}
\end{figure}
We give a lower bound for the number of the first type of these cycles to give an upper bound for the number of the second type of these cycles. There are $(k-1)^{h-1}$ vertices in $D_v$=$\{v_1,\ldots,v_{(k-1)^{h-1}}\}$. Each vertex $v_i$ has $y_i$ neighbours in $D_u$. Clearly, $\sum y_i=\Lambda$. In each vertex $v_i$, we have $\binom{y_i}{2}$ possible choices to form a girth-cycle through $u$ that does not contain the edge $uv$. Hence, the number of girth-cycles of the first type is
$$\sum_{i=1}^{(k-1)^{h-1}}\binom{y_i}{2}=\frac{1}{2}\sum_{i=1}^{(k-1)^{h-1}}y_i^2-\frac{1}{2}\sum_{i=1}^{(k-1)^{h-1}}y_i=\frac{1}{2}\sum_{i=1}^{(k-1)^{h-1}}y_i^2-\frac{\Lambda}{2}.$$
The inequality between the arithmetic and quadratic means for the degree set $\{y_1,\ldots,y_{(k-1)^{h-1}}\}$ gives a lower bound on the number of these cycles:
$$\frac{1}{2}\sum_{i=1}^{(k-1)^{h-1}}y_i^2-\frac{\Lambda}{2}\ge\frac{\Lambda^2}{2(k-1)^{(h-1)}}-\frac{\Lambda}{2}.$$
For small $\Lambda$ values, this lower bound is negative. Hence, we have the following lower bound:
$$\sum_{i=1}^{(k-1)^{h-1}}\binom{y_i}{2}\ge \max\left(0,\left\lceil\frac{\Lambda^2}{2(k-1)^{(h-1)}}-\frac{\Lambda}{2}\right\rceil\right).$$
Suppose that there are $m$ vertices in the set $M$. Their degree set is $\{x_1,\ldots,x_m\}$. We obtain a girth-cycle if we choose such a vertex $u_i$, its two neighbors, $w_1$ and $w_2$ in $D_u$ and their unique $(h-1)$-paths to the vertex $u$. Therefore the number of girth-cycles of the second type is exactly $\sum_{i=1}^m\binom{x_i}{2}.$ Now, we have the following upper bound on the number of these cycles:
$$\sum_{i=1}^m\binom{x_i}{2}=\lambda-\Lambda-\sum_{i=1}^{(k-1)^{h-1}}\binom{y_i}{2}\le\lambda-\Lambda- \max\left(0,\left\lceil\frac{\Lambda^2}{2(k-1)^{(h-1)}}-\frac{\Lambda}{2}\right\rceil\right).$$
We use this inequality to give a lower bound for $m$, but first, we need to rearrange the terms. We also use the fact that $\sum_{i=1}^mx_i=(k-1)^h-\Lambda$. 
Now, we have that 
\begin{align*}
\sum_{i=1}^mx_i^2=2\sum_{i=1}^m\binom{x_i}{2}+\sum_{i=1}^m x_i\le (k-1)^h+2\lambda-3\Lambda- 2\max\left(0,\left\lceil\frac{\Lambda^2}{2(k-1)^{(h-1)}}-\frac{\Lambda}{2}\right\rceil\right).
\end{align*}
By using the inequality between the arithmetic and quadratic mean, we get a lower bound for $m$:
$$m\ge\frac{\left(\sum_{i=1}^m x_i\right)^2}{\sum_{i=1}^m x_i^2}\ge \frac{\left((k-1)^h-\Lambda\right)^2}{2\lambda-3\Lambda+(k-1)^h-2\max\left(0,\left\lceil\frac{\Lambda^2}{2(k-1)^{(h-1)}}-\frac{\Lambda}{2}\right\rceil\right)}.$$
Since $G$ is a $k$-regular graph of girth $g$, it has at least $M(k,g)$ vertices, but with the lower bound of $m$, we also give a lower bound for the additional vertices. We add it to the Moore bound and obtain the generalization of the lower bound by Araujo-Pardo, Kiss, and Porupsánszki \cite{APKP23}.
\end{proof}
By definition, we do not have any assumption on the signature of a vertex-girth-regular graph. We only know the average of the signature that we can use to obtain a general lower bound on the order of vertex-girth-regular graphs.
\begin{theorem}
Let $G$ be a $vgr(n,k,g,\lambda)$-graph, where $g=2h$ is an even number. Then 
$$
    \begin{aligned}
        n\ge&2\frac{(k-1)^h-1}{k-2} \\
        &+\max_{\Lambda\in\left\{\left\lfloor\frac{2\lambda}{k}\right\rfloor,\left\lceil\frac{2\lambda}{k}\right\rceil\right\}}\left\lceil\frac{\left((k-1)^h-\Lambda\right)^2}{2\lambda-3\Lambda+(k-1)^h-2\max\left(0,\left\lceil\frac{\Lambda^2}{2(k-1)^{(h-1)}}-\frac{\Lambda}{2}\right\rceil\right)}\right\rceil.
    \end{aligned}
$$
\end{theorem}
\begin{proof}
Take the second term of the lower bound as a function of $\Lambda$:
$$f(x):=\frac{\left((k-1)^h-x\right)^2}{2\lambda-3x+(k-1)^h-2\max\left(0,\frac{x^2}{2(k-1)^{(h-1)}}-\frac{x}{2}\right)}.$$
We only need to check the behavior of $f$ around $ x = \frac{2\lambda}{k}$ by looking at its derivative. It is a straightforward task; therefore, we only present the final results. If $\lambda\ge\frac{k(k-1)^{h-1}}{2}$, then there is a local minimum at $\frac{2\lambda}{k}$. Otherwise, $f$ is a convex function. Hence, we can substitute the two closest integers to the average of the signature to obtain a lower bound on the order of vertex-girth-regular graphs.
\end{proof}
Finally, using spectral graph theory, we generalize a lower bound for the order of edge-girth-regular graphs that appeared in \cite{P23}. Consider the adjacency matrix of a $vgr(n,k,g,\lambda)$ and denote its eigenvalues by $\lambda_1\ge\ldots\ge\lambda_n$. Since the graph is $k$-regular, the largest eigenvalue is $k$. If the graph is bipartite, then the smallest eigenvalue is $-k$. Moreover, the sum of the $\ell$-th powers of the eigenvalues is the sum of the numbers of closed walks of length $\ell$ rooted at the vertices of the graph (summed through all vertices of the graph). We note that in a $k$-regular graph of girth $g$, the number of {\em cycle-free} closed walks of length $ \ell \leq g $ rooted at any vertex is independent of the choice of the vertex. We denote the number of cycle-free closed walks of length $\ell \leq g  $ rooted at a(ny)  vertex by $c(\ell,k)$, and observe that $c(\ell,k) = 0 $ for odd $ \ell \leq g  $. For even $\ell$, it is an $\frac{\ell}{2}$-th degree polynomial of $k$. For example, the first four $c(\ell,k)$ polynomials (for even lengths $\ell$) can be easily shown to be equal to the following:
$$
\begin{aligned}
c(2,k)=&k,\\
c(4,k)=&2k^2-k,\\
c(6,k)=&5k^3-6k^2+2k,\\
c(8,k)=&14k^4-28k^3+20k^2-5k.
\end{aligned}
$$
In particular, using the polynomials $ c(\ell,k)$, the 
number of closed walks of length $g/2$ in any $(k,g)$-graph is
$$\sum_{i=1}^n\lambda_i^{\frac{g}{2}}=n \cdot c\left(\frac{g}{2},k\right),$$
and the 
number of closed walks of length $g$ in a $vgr(n,k,g,\lambda)$-graph is equal to 
$$\sum_{i=1}^n\lambda_i^g=n \cdot \left(c(g,k)+2\lambda\right). $$
Next, we apply the inequality between the quadratic and arithmetic means to the set $\{\lambda_2^{g/2},\ldots,\lambda_n^{g/2}\}$ to obtain a lower bound on the order $n$ of a $vgr(n,k,g,\lambda)$-graph of even girth. For the bipartite case, we repeat the process for the set $\{\lambda_2^{g/2},\ldots,\lambda_{n-1}^{g/2}\}$ because $\lambda_n=-k.$ We obtain the following theorem, which is a direct generalization of Theorem 3.14 in \cite{P23}.
\begin{theorem}\label{lowerbound_pi}
Let $G$ be a $vgr(n,k,g,\lambda)$-graph, where $g$ is even.\\
If $g\equiv0$ (mod $4$), then
$$n(k,g,\lambda)\ge\frac{c(g,k)+2\lambda+k^{g}-2c(\frac{g}{2},k)k^{\frac{g}{2}}}{c(g,k)-c^2(\frac{g}{2},k)+2\lambda},$$
$$n_2(k,g,\lambda)\ge2\frac{c(g,k)+2\lambda+k^{g}-2c(\frac{g}{2},k)k^{\frac{g}{2}}}{c(g,k)-c^2(\frac{g}{2},k)+2\lambda}.$$
If $g\equiv2$ (mod $4$), then
$$n(k,g,\lambda)\ge\frac{c(g,k)+2\lambda+k^g}{c(g,k)+2\lambda},$$
$$n_2(k,g,\lambda)\ge\frac{2k^g}{c(g,k)+2\lambda}.$$
\end{theorem}
Finally, we remark that these lower bounds can often be improved by noticing that for a $vgr(v,k,g,\lambda)$-graph, $vk$ is even because of the handshaking lemma and $v\lambda$ must be a multiple of $g$ since each vertex is contained in precisely $\lambda$ girth-cycles and each girth-cycle contains precisely $g$ vertices. This leads to the following observation:
\begin{observation}
\label{obs:divisibility}
There are $\frac{v\lambda}{g}$ cycles of length $g$ in a $vgr(v,k,g,\lambda)$-graph.
\end{observation}
In other words, if $n(k,g,\lambda) \geq v$, but $v$ does not satisfy these conditions, then $n(k,g,\lambda) \geq v+1$ (and we can recursively apply this argument).  
\section{Non-existence results}\label{sec:nonExistence}

Our first result in this section is a generalization of Lemma~\ref{easy} and is reminiscent of a similar result proven in 
\cite{KMS22} for $k$-regular graphs in which all vertices have the same signature; and hence also for edge-girth-regular graphs. Since vertex-girth-regular graphs are not necessarily signature-regular, our result requires an independent proof which is based on ideas different from those in \cite{KMS22}. Moreover, we prove both the odd and the even girth cases, whereas~\cite{KMS22} only covers the even girth case. 

For $ k \geq 3 $ and $ g = 2s+1 \geq 3 $, let $ nc(k,g) = \frac{k(k-1)^s}{2} $. Then, the number of girth-cycles through any vertex of a
$ (k,g) $-graph is at most $ nc(k,g) $, and this number is equal to $ nc(k,g) $ if and only if the graph is a $(k,g)$-Moore graph (see Proposition~\ref{prop:elementalProperties}).

\begin{theorem}\label{non-exist-odd}
Let $ k \geq 3 $, $ g = 2s+1 \geq 7 $, and $ 0 < \epsilon \leq \frac{k-1}{2} $ be integers. Then there is no $vgr(n,k,g,nc(k,g)- \epsilon) $-graph.
\end{theorem}

\begin{proof}
Suppose $ k,g $ and $ \epsilon $ satisfy the above requirements, and, by means of contradiction, let us assume that 
$\Gamma$ is a $vgr(k,g,nc(k,g)-\epsilon) $. Let $ u $ be an arbitrary vertex of $\Gamma$, and let $ \Gamma_u^s $
be the subgraph of $\Gamma$ induced by the union of the sets $ N_{\Gamma}(u,i) = \{ v \in V(\Gamma) \; | \; 
d_{\Gamma}(u,v)=i \} $, $ 0 \leq i \leq s $. Let us call the edges of $ \Gamma_u^s $ connecting any two vertices of distance $s$
from $u$ {\em horizontal} and note that the number of $g$-cycles through $u$ is equal to the number of these horizontal
edges. Since the number of $g$-cycles through $u$ is equal to $ nc(k,g)-\epsilon $ and $nc(k,g)$ is the number of pairs
of distinct vertices of distance $s$ from $u$ that could potentially have an edge between them, there must exist exactly $\epsilon$ pairs of these vertices that are not adjacent 
in $\Gamma$. Since each vertex of $\Gamma$ is assumed to be of degree $k$, this means that each pair of non-adjacent
vertices $v_1, v_2 $ of distance $s$ from $u$ gives rise to two edges connecting a vertex of distance $s$ from
$u$, namely $v_1$ and $v_2$, to a vertex of distance $s+1$ from $u$. Thus, $\Gamma$ contains exactly $2 \epsilon$ edges 
between $ N_{\Gamma}(u,s) $ and $ N_{\Gamma}(u,s+1) $. To simplify our arguments, let us call the vertices of 
$ \Gamma_u^s $ black, and the rest of the vertices of $ \Gamma $ white. Thus, $\Gamma$ contains exactly $2 \epsilon$ 
black-white edges (i.e., having different colored endpoints) and all the other edges are either black (both ends are black) or white
(both ends are white). Let $v$ be a black vertex adjacent to a black-white edge, and consider the induced subgraph $ \Gamma_v^s $
induced by the union of the sets $ N_{\Gamma}(v,i) $, $ 0 \leq i \leq s $. Once again, the number of $g$-cycles containing
$v$ is equal to the number of horizontal edges (i.e., edges connecting two vertices from  $ N_{\Gamma}(v,s) $ in 
$ \Gamma_v^s $). Let $ w_1,w_2,\ldots,w_k $ be the neighbors of $v$. Let $ \Gamma_{v,w_i}^s $, $ 1 \leq i \leq k $,  denote the induced subgraphs of $ \Gamma_v^s $ `rooted' at the vertices $w_i$ which are disjoint subgraphs induced by the unions of vertices $ N_{\Gamma}(v,j) \cap N_{\Gamma}(w_i,j-1) $, 
$ 1 \leq j \leq s $ (in other words, the subgraphs consist of $w_i$ and vertices of distance at least $2$ and at most $s$ from $v$ whose shortest path toward $v$ contains $w_i$). We may assume without loss of generality that the vertices $ w_1,w_2,\ldots,w_{\ell} $, $ \ell \geq 1 $, are 
white (since $v$ was chosen to be incident with at least one black-white edge), and the 
vertices $ w_{\ell +1}, w_{\ell +2}, \ldots, w_k $ are black, where $ 1 \leq \ell < k $ (since $ 2 \epsilon \leq k-1 $ and thus, at least one of the neighbors of $v$ is black). 

To complete our argument, 
let us consider the neighbors $ w_{1,1},w_{1,2}, \ldots, w_{1,k-1} $ of the white vertex $ w_1 $ distinct from $v$. Each of 
the vertices $ w_{1,1},w_{1,2}, \ldots, w_{1,k-1} $ determines an induced subgraph $ \Gamma_{v,w_1,w_{1,i}}^s $ of $ \Gamma_{v,w_1}^s $ consisting of $w_1$, the edge $w_1w_{1,i}$, and the subgraph of $ \Gamma_{v,w_1}^s $ induced by the subset of vertices of $ \Gamma_{v,w_1}^s $ comprised of vertices of distance at least $2$ from $v$ whose shortest path to
$v$ contains $w_{1,i} $. Any two of these subgraphs share exactly one vertex; the vertex $w_1$.

There are the total of $2 \epsilon \leq k-1 $ black-white edges contained in $\Gamma$. Let us 
assume that there are $r$ subgraphs among the $ \Gamma_{v,w_1,w_{1,i}}^s $, $ 1 \leq i \leq k-1 $, which contain at least one black-white edge.
Note that $ r < k-1 $ as otherwise the subgraphs would contain at least $k-1$ black-white edges (recall that
they only share a vertex) which together with the black-white $vw_1$ would make for $k$ black-white edges, while we assume 
that the number of black-white edges does not exceed $k-1$. Next, let us apply the same kind of argument to the $k-1$ subgraphs $ \Gamma_{v,w_i}^s $, $ 2 \leq i \leq k $. Since $v$ is black and $vw_1$ is black-white, at least one of these
subgraphs does not contain a black-white edge, and therefore consists entirely of black vertices. 
Let us denote the number of subgraphs $ \Gamma_{v,w_i}^s $, 
$ 2 \leq i \leq k $ that do contain at least one black-white edge by $t$, we have argued that $ 0 \leq t < k-1 $.
Now, consider the $k-1-r > 0 $ subgraphs $ \Gamma_{v,w_1,w_{1,i}}^s $, $ 1 \leq i \leq k-1 $, that {\em do not contain}
a black-white edge. Since $w_1$ is white and it is contained in each of these subgraphs, those subgraphs that do not contain
a black-white edge must consist entirely of white vertices. Each of them contains $(k-1)^{s-2}$ vertices of distance $s$ from
$v$, and hence there are $(k-1-r)(k-1)^{s-2}$ white vertices of distance $s$ from $v$ whose shortest path to $v$ contains 
$w_1$. Since the girth of $\Gamma$ is assumed to be equal to $2s+1 \geq 7$, any horizontal edge emanating from these
white vertices has to connect them to exactly one of the subgraphs $ \Gamma_{v,w_i}^s $, $ 2 \leq i \leq k $, with $t \geq 1$ of them consisting entirely of black vertices. This means in particular, that all the potential $t(k-1-r)(k-1)^{s-2}$ horizontal edges (with respect to $v$) connecting the $(k-1-r)(k-1)^{s-2}$ white vertices to the black $ \Gamma_{v,w_i}^s $ would have to be
black-white. Recall that we assume that $ 2s+1 \geq 7 $ or that $ s \geq 3 $ and hence $s-2 \geq 1$. Since $\Gamma$ 
contains exactly $2\epsilon$ black-white edges, one of which is the edge $vw_1$, and the subgraphs 
$ \Gamma_{v,w_1,w_{1,i}}^s $, $ 1 \leq i \leq k-1 $, and $ \Gamma_{v,w_i}^s $, $ 2 \leq i \leq k $, contain at least $r+t$
black-white edges, there at most $2\epsilon - r - t - 1$ black-white edges to be used between $ \Gamma_{v,w_1}^s $ and
$ \Gamma_{v,w_i}^s $, $ 2 \leq i \leq k $. This means that of the $t(k-1-r)(k-1)^{s-2}$ potential horizontal black-white edges between $ \Gamma_{v,w_1}^s $ and $ \Gamma_{v,w_i}^s $, $ 2 \leq i \leq k $, {\em at least} 
$t(k-1-r)(k-1)^{s-2} - (2\epsilon - r - t - 1)$ are not edges of $\Gamma$. Hence, $\Gamma$ misses at least this many edges that would be horizontal edges with respect to $v$.  Since $ s \geq 3 $ and $ t \geq 1 $, we obtain
\[ t(k-1-r)(k-1)^{s-2} - (2\epsilon - r - t - 1) \geq (k-1-r)(k-1) -2 \epsilon + r + 1 .\]
Furthermore, since $ k \geq 3 $, the value $ (k-1-r)(k-1) -2 \epsilon + r + 1 $ is minimized as a function of $r$ when 
$r$ is maximal possible, (i.e., when $ r = k-2 $). This yields:
\[ (k-1-r)(k-1) -2 \epsilon + r + 1 \geq (k-1-(k-2))(k-1) -2 \epsilon + (k-2) + 1 = 2(k-1) - 2 \epsilon .\]
The assumption $ \epsilon \leq \frac{k-1}{2} $ yields then that
\[ 2(k-1) - 2 \epsilon \geq k-1 > \epsilon, \]
which means that the number of missing horizontal edges with respect to $v$ is bigger than $\epsilon$. The final 
contradiction now follows from the fact that the number of girth-cycles through $v$ is the number of horizontal edges
with respect to $v$ which is smaller than $ nc(k,g)- \epsilon $, which contradicts the assumption that $\Gamma$ is a
$vgr(n,k,g,nc(k,g)- \epsilon)$-graph.
\end{proof}

Next, consider $ k \geq 3 $ and $ g = 2s \geq 4 $, and let $ nc(k,g) = \frac{k(k-1)^s}{2} $. Once again, the number of girth-cycles through any vertex of a
$ (k,g) $-graph is at most $ nc(k,g) $, and this number is equal to $ nc(k,g) $ if and only if the graph is a $(k,g)$-Moore graph.

\begin{theorem}\label{non-exist-even}
Let $ k \geq 3 $, $ g = 2s \geq 4 $, and $ 0 < \epsilon < k-1 $ be integers. Then there is no $vgr(n,k,g,nc(k,g)- \epsilon) $-graph.
\end{theorem}
\begin{proof}
In similarity to the proof of Proposition~\ref{prop:elementalProperties}, the argument for the even girth case is quite different from the argument used in the proof of Theorem~\ref{non-exist-odd}. For the sake of obtaining a contradiction, let $G$ be a $vgr(n,k,g,nc(k,g)- \epsilon) $-graph,
$ 0 < \epsilon < k-1 $, let $u_1u_2$ be an edge of $G$ and let the sets of girth-cycles $\mathcal{A}_{u_1}$, $\mathcal{B}_{u_1}$ and $\mathcal{C}_{u_1}$ be defined as in Proposition~\ref{prop:elementalProperties}. Since $\epsilon>0$, at least one horizontal edge is missing and therefore $|\mathcal{A}_{u_1}|<(k-1)^s$
and $|\mathcal{B}_{u_1}|\leq ((k-1)^{s-1}-1){k-1 \choose 2}+{k-2 \choose 2}$. Since every additional missing horizonal edge decreases both $|\mathcal{A}_{u_1}|$ and $|\mathcal{B}_{u_1}|$ by more than it (potentially) increases $|\mathcal{C}_{u_1}| $, the maximum of the sum $|\mathcal{A}_{u_1}|+|\mathcal{B}_{u_1}|+|\mathcal{C}_{u_1}|$ is attained when there is exactly one missing horizontal edge, in which case, $|\mathcal{A}_{u_1}|=(k-1)^s-1$, $|\mathcal{B}_{u_1}|=((k-1)^{s-1}-1){k-1 \choose 2}+{k-2 \choose 2}$, and $|\mathcal{C}_{u_1}|=0$. 
Thus, the number of girth-cycles containing $u_1$ (i.e., the sum $|\mathcal{A}_{u_1}|+|\mathcal{B}_{u_1}|+|\mathcal{C}_{u_1}|$) is bounded from above by 
$ (k-1)^s-1 + ((k-1)^{s-1}-1){k-1 \choose 2}+{k-2 \choose 2} = nc(k,g)-(k-1)$. This leads to a contradiction with the assumption that $\epsilon < k-1$,
and therefore no such $G$ can exist.
\end{proof}

Interestingly, we could not prove Theorem~\ref{non-exist-odd} for the case $s=2$. Consulting the tables at the end of our article suggests that no $ vgr(n,3,5,\lambda)$-graphs exist for $ \lambda = 4 $ or 
$ \lambda = 5 $, i.e., for $ \epsilon = 1 $ or $ \epsilon = 2 $ (where $ 1 \leq \frac{3-1}{2} $). This suggests that
Theorem~\ref{non-exist-odd} might hold for $s=2$ (i.e., $g=5$) as well. In addition, Observation~\ref{obs:EgrToVgr} suggests that the theorem may even hold for $s=1$ (i.e., $ g = 3 $).
Maybe even more intriguingly, inspecting our computational results summarized at the end of this article opens space for further non-existence results for cases when $\lambda$
is not particularly close to $ nc(k,g) $. Similar results 
exist in case of edge-girth-regular graphs, where, for 
example, it has been shown in \cite{D24} that no 
$(3,7,6)$ and no $ (3,8,14) $ edge-girth-regular graphs exist. This raises the possibility of finding further arithmetic conditions that would
yield the non-existence of $ vgr(n,k,g,\lambda)$-graphs.
Having brought up the connections between edge-girth-regular, girth-regular,
and vertex-girth-regular graphs, we also wish to point out
that non-existence results concerning vertex-girth-regular 
graphs yield the non-existence of corresponding edge-girth-regular and girth-regular graphs as well.

\section{Exhaustive generation algorithm}\label{sec:generationAlgo}
Goedgebeur and the second author~\cite{GJ24} described an algorithm to exhaustively generate all $egr(v,k,g,\lambda)$-graphs for given integers $v$, $k$, $g$ and $\lambda$. In the current paper, we adapt this algorithm to generate all $vgr(v,k,g,\lambda)$-graphs and add a different heuristic and pruning rule that speed up the algorithm without affecting the exhaustiveness guarantee.

The algorithm (pseudo code shown in Algorithm~\ref{algo:recAddEdges} and~\ref{algo:genAlgo}) expects as input four integers $v$, $k$, $g$ and $\lambda$ and works as follows: the algorithm starts from a $(k,g)$ Moore tree and adds isolated vertices until there are $v$ vertices in total (note that this graph occurs as a subgraph of every $vgr(v,k,g,\lambda)$-graph). The algorithm then recursively adds edges to this graph in all possible ways such that no $vgr(v,k,g,\lambda)$-graphs are excluded from the search space. 

\begin{algorithm}[ht!]
\caption{recursivelyAddEdges(Integer $v$, Integer $k$, Integer $g$, Integer $\lambda$, Graph $G=(V,E)$, Set validEdgesToBeAdded)}
\label{algo:recAddEdges}
  \begin{algorithmic}[1]
		\STATE // Each call adds one edge to $G$
        \IF{One of the pruning rules can be applied}
		  \RETURN
        \ENDIF
		\STATE // $G$ has the right number of edges
		\IF{$|E|=\frac{vk}{2}$}
			\IF{$G$ is a $vgr(v,k,g,\lambda)$-graph}
				\STATE Output $G$
			\ENDIF
			\RETURN
        \ENDIF
        \STATE // Apply heuristic for choosing the next edge $e$ to consider
        \STATE $u \gets \arg\min_{w \in V(G), deg_G(w)<k}(|\{e \in \text{validEdgesToBeAdded and }e\text{ incident with }w\}|-(k-deg_G(w)))$
        \STATE $e \gets \text{arbitrary edge from validEdgesToBeAdded incident with }u$
        \STATE // Option 1: add this edge to $G$
        \STATE $G' \gets (V,E\cup\{e\})$
        \STATE $\text{newValidEdgesToBeAdded} \gets \text{update}(\text{validEdgesToBeAdded},G')$
        \STATE $\text{recursivelyAddEdges}(v,k,g,\lambda,G',\text{newValidEdgesToBeAdded})$
        \STATE // Option 2: do not add this edge to $G$
        \STATE $\text{recursivelyAddEdges}(v,k,g,\lambda,G,\text{validEdgesToBeAdded}\setminus\{e\})$
  \end{algorithmic}
\end{algorithm}

\begin{algorithm}[ht!]
\caption{generateAllVertexGirthRegularGraphs(Integer $v$, Integer $k$, Integer $g$, Integer $\lambda$)}
\label{algo:genAlgo}
  \begin{algorithmic}[1]
		\STATE // This function generates all $vgr(v,k,g,\lambda)$-graphs
		\STATE $T \gets (k,g)\text{ Moore tree}$
		\STATE // There are no $vgr(v,k,g,\lambda)$-graphs if $v$ is too small
		\IF{$v < |V(T)|$}
			\RETURN
            \ENDIF
		\STATE // Add $v-|V(T)|$ isolated vertices
		\STATE $G \gets \text{addIsolatedVertices}(T,v-|V(T)|)$
		\STATE $\text{validEdgesToBeAdded} \gets \text{calculateValidEdgesToBeAdded}(G)$
		\STATE $\text{recursivelyAddEdges}(v,k,g,\lambda,G,\text{validEdgesToBeAdded})$
  \end{algorithmic}
\end{algorithm}

In order to obtain an efficient algorithm, a heuristic is used for the order in which the edges are added as well as several pruning rules that allow the algorithm to backtrack as soon as a graph is encountered for which the algorithm can decide that it cannot occur as a subgraph of any $vgr(v,k,g,\lambda)$-graph. For each edge that the algorithm considers, it will branch into two possibilities (adding the edge in the first branch and not adding the edge in the second branch) and the algorithm keeps track of which edges can potentially be added to the graph. The new heuristic that the algorithm employs is to choose the next edge to consider as an edge which is incident to the vertex $u$ for which the difference between the number of potential valid edges that could be added incident with $u$ and $k$ minus the degree of $u$ is minimized. For example, if $k=5$ and $u$ has degree $3$ and there are only $2$ possibilities for edges that could potentially be added incident with $u$, then the difference will be equal to $0$ and the algorithm will add these two missing edges as soon as possible. This heuristic works well, because these two edges must be added eventually and adding them sooner further constrains the search space. After each iteration, the algorithm will mark potential edges to add as invalid if their addition would result in either:
\begin{itemize}
    \item a graph with girth smaller than $g$;
    \item a graph in which some vertex has degree larger than $k$;
    \item a graph in which some vertex is contained in more than $\lambda$ girth-cycles.
\end{itemize}

In each of these cases, the resulting graph can be excluded from the search space, because adding more edges will never lead to a $vgr(v,k,g,\lambda)$-graph.

Additionally, the algorithm prunes the current graph if it is isomorphic with a graph that was encountered previously during the search, because this cannot result in any new $vgr(v,k,g,\lambda)$-graphs which were not previously generated. Moreover, the algorithm prunes the current graph if it has a vertex $u$ of degree less than $k$ for which the number of potential edges that can be added incident with $u$ is strictly smaller than $k$ minus the degree of $u$, because this means that $u$ cannot obtain degree $k$ anymore. Finally, we also employ a new pruning rule based on the following proposition.

\begin{proposition}
Let $G$ be a graph and let $cyc(G,g,u)$ denote the number of cycles in $G$ of length $g$ containing vertex $u$. If, for given integers $g$, $k$ and $\lambda$, there is a vertex $u \in V(G)$ with degree $k$ such that $$(k-2)\lambda+2cyc(G,g,u)-\sum_{u' \in N_G(u)} cyc(G,g,u')<0$$ or simultaneously $$(2-k)\lambda+\sum_{u' \in N_G(u)} cyc(G,g,u')-2\max_{u' \in N_G(u)}(cyc(G,g,u')) \geq 0$$ and $$(k-2)\lambda+cyc(G,g,u)-\sum_{u' \in N_G(u)} cyc(G,g,u')+\max_{u' \in N_G(u)}(cyc(G,g,u'))<0,$$ then $G$ does not occur as a subgraph of any $vgr(v,k,g,\lambda)$-graph.
\end{proposition}
\begin{proof}
Suppose for the sake of obtaining a contradiction that $u \in V(G)$ is vertex with degree $k$ satisfying the first or the second condition from the proposition and $G$ occurs as a subgraph of a $vgr(v,k,g,\lambda)$-graph $G'$. There are precisely $\lambda-cyc(G,g,u)$ cycles of length $g$ in $G'$ which contain $u$ and contain some edge from $E(G') \setminus E(G)$ (i.e., they are not cycles of $G$). Clearly, each of these cycles contains at least two neighbors of $u$. Since each vertex in $G'$ is contained in precisely $\lambda$ girth-cycles, this implies that $2(\lambda-cyc(G,g,u)) \leq \sum_{u' \in N_G(u)}(\lambda-cyc(G,g,u'))$ and thus $(k-2)\lambda+2cyc(G,g,u)-\sum_{u' \in N_G(u)} cyc(G,g,u') \geq 0$. Hence, if the first condition of the lemma is satisfied, we immediately obtain a contradiction. So let us assume that the second condition holds. Since
$(2-k)\lambda+\sum_{u' \in N_G(u)} cyc(G,g,u')-2\max_{u' \in N_G(u)}(cyc(G,g,u')) \geq 0$, we have
\begin{align*}
    \max_{u' \in N_G(u)}(\lambda-cyc(G,g,u'))\\
    \geq \sum_{u' \in N_G(u)} (\lambda-cyc(G,g,u'))-\max_{u' \in N_G(u)}(\lambda-cyc(G,g,u')).
\end{align*}
Since each girth-cycle containing $u$ in $G'$ contains at least two neighbors of $u$, this in turn implies that there are at most $\sum_{u' \in N_G(u)} (\lambda-cyc(G,g,u'))-\max_{u' \in N_G(u)}(\lambda-cyc(G,g,u'))$ girth-cycles in $G'$ containing $u$, which are not cycles of $G$. Hence, we have $\lambda-cyc(G,g,u) \leq \sum_{u' \in N_G(u)} (\lambda-cyc(G,g,u'))-\max_{u' \in N_G(u)}(\lambda-cyc(G,g,u'))$ and so
$$(k-2)\lambda+cyc(G,g,u)-\sum_{u' \in N_G(u)} cyc(G,g,u')+\max_{u' \in N_G(u)}(cyc(G,g,u')) \geq 0.$$
We again obtain a contraction by assuming that the second condition holds.
\end{proof}

\section{Computational lower and upper bounds for $n(k,g,\lambda)$}
We implemented the algorithm from Section~\ref{sec:generationAlgo} to exhaustively generate vertex-girth-regular graphs and used it to obtain lower and upper bounds for $n(k,g,\lambda)$. More specifically, if the algorithm does not generate any $vgr(v',k,g,\lambda)$-graphs for all $v'<v$, then $n(k,g,\lambda) \geq v$ since the algorithm is exhaustive. On the other hand, if the algorithm generates at least one $vgr(v,k,g,\lambda)$-graph, then clearly $n(k,g,\lambda) \leq v$. We ran the algorithm for $k=3$, $3 \leq g \leq 8$ and $k=4$, $3 \leq g \leq 6$. The bounds that we obtained from the computations described in this section are summarized in Tables~\ref{tab:comparison3Reg} (cubic case) and~\ref{tab:comparison4Reg} (quartic case). Bold values indicate cases where the lower bound is equal to the upper bound.

Apart from running the algorithm from Section~\ref{sec:generationAlgo}, we also ran the algorithm GENREG~\cite{Me99} for generating all connected $k$-regular graphs of girth $g$ on $v$ vertices and filtered out those graphs which are vertex-girth-regular. The lower bounds that we obtained in this way could sometimes also be further improved by applying Observation~\ref{obs:divisibility} (i.e., the order $v$ must be such that $v\lambda$ is a multiple of the girth $g$). Moreover, we also applied the lower bounds from Section~\ref{sec:LBounds} and filled the table with best lower bounds among the previously discussed methods.

For improving the upper bounds in cases where our algorithm and GENREG were unable to find any graphs, we also filtered out vertex-girth-regular graphs from known lists of regular graphs (with additional symmetry properties). More specifically, we consulted the list of all vertex-transitive graphs until order 47~\cite{HR20}, cubic vertex-transitive graphs until order 1280~\cite{PSV13}, cubic arc-transitive graphs until order 2048~\cite{CD02} and quartic arc-transitive graphs until order 640~\cite{PSV13,PSV15}. Additionally, we also applied the construction from Proposition~\ref{prop:generalizedTruncation} to obtain upper bounds. The total CPU-time for all computations in this paper amounts to approximately 2 CPU-years (the computations were executed on the hardware of the Flemish Supercomputer Center). We make all code and data related to this paper publicly available at \url{https://github.com/JorikJooken/vertexGirthRegularGraphs}. The graphs that we found can also be downloaded from the House of Graphs~\cite{CDG23} by searching for the term “vertex-girth-regular”.

We now briefly discuss some observations and remarkable graphs that we found based on these computations. From Tables~\ref{tab:comparison3Reg} and~\ref{tab:comparison4Reg} it is clear that often $n(k,g,\lambda) \geq n(k,g,\lambda+1)$, as one could intuitively expect (e.g. Moore graphs occur for the maximal value of $\lambda$). However, this inequality does not always hold. For example, we showed that $n(3,8,8)=42<n(3,8,9)=48$. The corresponding graphs achieving these bounds are shown in Fig.~\ref{fig:42and48Vertices}. These two graphs are also among the largest graphs for which we were able to prove that they are extremal.

\begin{figure}[h!]
\begin{center}
\begin{tikzpicture}[scale=0.95]
  \def\sides{42}
  \def\radius{3}

  \foreach \i in {1,...,\sides} {
    \fill ({360/\sides * \i}:\radius) circle (1.5pt);
  }
  
  \foreach \i in {1,4,7,...,40} {
    \draw ({360/\sides * (\i + 1)}:\radius) -- ({360/\sides * \i}:\radius);
    \draw ({360/\sides * (\i + 8)}:\radius) -- ({360/\sides * \i}:\radius);
    \draw ({360/\sides * (\i + 41)}:\radius) -- ({360/\sides * \i}:\radius);
  }
  
    \foreach \i in {2,5,8,...,41} {
    \draw ({360/\sides * (\i + 1)}:\radius) -- ({360/\sides * \i}:\radius);
    \draw ({360/\sides * (\i + 21)}:\radius) -- ({360/\sides * \i}:\radius);
    \draw ({360/\sides * (\i + 41)}:\radius) -- ({360/\sides * \i}:\radius);
  }

    \foreach \i in {3,6,9,...,42} {
    \draw ({360/\sides * (\i + 1)}:\radius) -- ({360/\sides * \i}:\radius);
    \draw ({360/\sides * (\i + 34)}:\radius) -- ({360/\sides * \i}:\radius);
    \draw ({360/\sides * (\i + 41)}:\radius) -- ({360/\sides * \i}:\radius);
  }
\end{tikzpicture} \quad
\begin{tikzpicture}[scale=0.95]
  \def\sides{48}
  \def\radius{3}

  \foreach \i in {1,...,\sides} {
    \fill ({360/\sides * \i}:\radius) circle (1.5pt);
  }
  
  \foreach \i in {1,5,9,...,45} {
    \draw ({360/\sides * (\i + 1)}:\radius) -- ({360/\sides * \i}:\radius);
    \draw ({360/\sides * (\i + 29)}:\radius) -- ({360/\sides * \i}:\radius);
    \draw ({360/\sides * (\i + 47)}:\radius) -- ({360/\sides * \i}:\radius);
  }

  \foreach \i in {2,6,10,...,46} {
    \draw ({360/\sides * (\i + 1)}:\radius) -- ({360/\sides * \i}:\radius);
    \draw ({360/\sides * (\i + 19)}:\radius) -- ({360/\sides * \i}:\radius);
    \draw ({360/\sides * (\i + 47)}:\radius) -- ({360/\sides * \i}:\radius);
  }

  \foreach \i in {3,7,11,...,47} {
    \draw ({360/\sides * (\i + 1)}:\radius) -- ({360/\sides * \i}:\radius);
    \draw ({360/\sides * (\i + 41)}:\radius) -- ({360/\sides * \i}:\radius);
    \draw ({360/\sides * (\i + 47)}:\radius) -- ({360/\sides * \i}:\radius);
  }

  \foreach \i in {4,8,12,...,48} {
    \draw ({360/\sides * (\i + 1)}:\radius) -- ({360/\sides * \i}:\radius);
    \draw ({360/\sides * (\i + 7)}:\radius) -- ({360/\sides * \i}:\radius);
    \draw ({360/\sides * (\i + 47)}:\radius) -- ({360/\sides * \i}:\radius);
  }
\end{tikzpicture}
\end{center}
\caption{An extremal $vgr(42,3,8,8)$-graph (left) and an extremal $vgr(48,3,8,9)$-graph (right).}\label{fig:42and48Vertices} 
\end{figure}
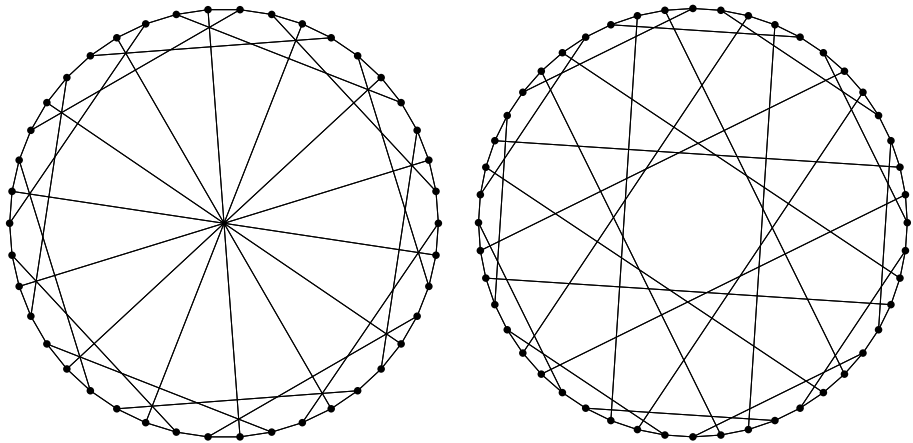

We also remark that several famous graphs appear as (extremal) vertex-girth-regular graphs. Among others we mention all Moore graphs and several cages~\cite{EJ08}, the Platonic solids, several incidence graphs, the Pappus graph (a $vgr(18,3,6,6)$-graph), the Coxeter graph (a $vgr(28,3,7,6)$-graph), the burnt pancake graph BP(3) (a $vgr(48,3,8,6)$-graph) and the generalized Petersen graph $G(13,5)$, which is a $vgr(26,3,7,7)$-graph (see Fig.~\ref{fig:generalisedPetersen}) and appears for example in~\cite{CGJ23} as the cubic graph of girth 7 on 26 vertices with the most connected induced subgraphs among all such graphs. 
\begin{figure}[ht]
\centering
\begin{tikzpicture}[scale=0.65]
\foreach \x in {0,1,...,12}{
\draw[fill] (\x*360/13:4.5) circle (1.8pt);
\draw[fill] (\x*360/13:3) circle (1.8pt);
\draw (\x*360/13:3)--(\x*360/13:4.5);
\draw (\x*360/13+360/13:4.5)--(\x*360/13:4.5);
\draw (\x*360/13+1800/13:3)--(\x*360/13:3);
}
\end{tikzpicture}
\caption{The generalized Petersen graph $G(13,5)$ is a $vgr(26,3,7,7)$-graph.}\label{fig:generalisedPetersen}
\end{figure}
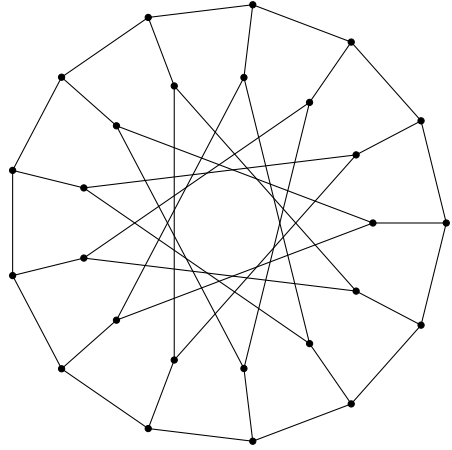

For the cage problem, all known cages of even girth are bipartite and an important open question asks whether this is always the case~\cite{EJ08}. We remark that there are several non-bipartite extremal vertex-girth-regular graphs of even girth (e.g. the $vgr(42,3,8,8)$-graph shown in Fig.~\ref{fig:42and48Vertices} is such an example). Hence, the analogous question for vertex-girth-regular graphs has a negative answer. In fact there are plenty of non-bipartite extremal vertex-girth-regular graphs of even girth, for example 25 pairwise non-isomorphic $vgr(20,4,4,1)$-graphs.

In Section~\ref{sec:introduction} we mentioned that several subclasses of vertex-girth-regular graphs have received attention before in the literature (e.g. vertex-transitive graphs~\cite{HR20}, edge-girth-regular graphs \cite{DFJR21,JKM18} and graphs in which each vertex has the same signature \cite{PV19}). We remark that many of the extremal vertex-girth-regular graphs that we found do not belong to any of these subclasses. For example, the graph shown in Fig.~\ref{fig:25Vertices} is an extremal vertex-girth-regular graph in which 20 vertices have signature $\{5,4,4,3\}$ and 5 vertices have signature $\{4,4,4,4\}$. In total, we found 98 extremal vertex-girth-regular graphs which were not vertex-transitive (91 of these were not edge-girth-regular and 47 of these contained vertices with different signatures). Another interesting example to mention is again related to the cage problem. The smallest known cubic graph of girth 13 has 272 vertices~\cite{EJ08}. This graph is not edge-girth-regular, but it is vertex-girth-regular (and in fact even vertex-transitive; every vertex has signature $\{18,18,16\}$).
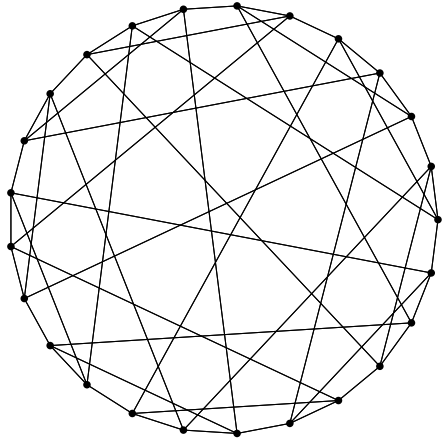
\begin{figure}[h!]
\begin{center}
\begin{tikzpicture}[scale=0.95]
  \def\sides{25}
  \def\radius{3}

  \foreach \i in {1,...,\sides} {
    \fill ({360/\sides * \i}:\radius) circle (1.5pt);
  }
  
  \foreach \i in {1,6,...,21} {
    \draw ({360/\sides * (\i + 1)}:\radius) -- ({360/\sides * \i}:\radius);
    \draw ({360/\sides * (\i + 17)}:\radius) -- ({360/\sides * \i}:\radius);
    \draw ({360/\sides * (\i + 21)}:\radius) -- ({360/\sides * \i}:\radius);
    \draw ({360/\sides * (\i + 24)}:\radius) -- ({360/\sides * \i}:\radius);
  }

  \foreach \i in {2,7,...,22} {
    \draw ({360/\sides * (\i + 1)}:\radius) -- ({360/\sides * \i}:\radius);
    \draw ({360/\sides * (\i + 4)}:\radius) -- ({360/\sides * \i}:\radius);
    \draw ({360/\sides * (\i + 12)}:\radius) -- ({360/\sides * \i}:\radius);
    \draw ({360/\sides * (\i + 24)}:\radius) -- ({360/\sides * \i}:\radius);
  }

  \foreach \i in {3,8,...,23} {
    \draw ({360/\sides * (\i + 1)}:\radius) -- ({360/\sides * \i}:\radius);
    \draw ({360/\sides * (\i + 8)}:\radius) -- ({360/\sides * \i}:\radius);
    \draw ({360/\sides * (\i + 17)}:\radius) -- ({360/\sides * \i}:\radius);
    \draw ({360/\sides * (\i + 24)}:\radius) -- ({360/\sides * \i}:\radius);
  }

  \foreach \i in {4,9,...,24} {
    \draw ({360/\sides * (\i + 1)}:\radius) -- ({360/\sides * \i}:\radius);
    \draw ({360/\sides * (\i + 13)}:\radius) -- ({360/\sides * \i}:\radius);
    \draw ({360/\sides * (\i + 21)}:\radius) -- ({360/\sides * \i}:\radius);
    \draw ({360/\sides * (\i + 24)}:\radius) -- ({360/\sides * \i}:\radius);
  }

  \foreach \i in {5,10,...,25} {
    \draw ({360/\sides * (\i + 1)}:\radius) -- ({360/\sides * \i}:\radius);
    \draw ({360/\sides * (\i + 4)}:\radius) -- ({360/\sides * \i}:\radius);
    \draw ({360/\sides * (\i + 8)}:\radius) -- ({360/\sides * \i}:\radius);
    \draw ({360/\sides * (\i + 24)}:\radius) -- ({360/\sides * \i}:\radius);
  }
\end{tikzpicture} 
\end{center}
\caption{An extremal $vgr(25,4,5,8)$-graph for which 20 vertices have signature $\{5,4,4,3\}$ and 5 vertices have signature $\{4,4,4,4\}$.}\label{fig:25Vertices} 
\end{figure}

\subsection{Sanity checks}
The bounds obtained in this section rely on the outcome of the author's implementation of the algorithm from Section~\ref{sec:generationAlgo}. Therefore, it is very important to take extra measures to ensure the correctness of our implementation, since an incorrect implementation would invalidate the obtained bounds. We first remark that, as expected, the bounds obtained by all different methods described in the previous paragraphs are in complete agreement with each other. Moreover, we also compared the outcome of our algorithm with the outcome of filtering the vertex-girth-regular graphs from all graphs produced by GENREG for orders that are larger than $n(k,g,\lambda)$ (where often there are many $vgr(v,k,g,\lambda)$-graphs for a fixed $v$) and obtained exactly the same graphs in each case. We were also able to independently find many graphs that belong to subclasses of vertex-girth-regular that received attention in the literature before~\cite{GJ24,PV19}. 

\begin{table}[htbp]
    \centering
    \begin{subtable}{0.40\textwidth}
        \centering
 \renewcommand{\arraystretch}{0.9}
    {\footnotesize
    \begin{tabular}{|c| c | c | c | c |}
        \hline
        $k$ & $g$ & $\lambda$ & $n(k,g,\lambda) \geq$ & $n(k,g,\lambda) \leq$\\
        \hline
        3&3&1&$\mathbf{6}$&$\mathbf{6}$\\
        3&3&2&$\bm{\infty}$&$\bm{\infty}$\\
        3&3&3&$\mathbf{4}$&$\mathbf{4}$\\
        \hline
        3&4&1&$\mathbf{12}$&$\mathbf{12}$\\
        3&4&2&$\mathbf{8}$&$\mathbf{8}$\\
        3&4&3&$\mathbf{8}$&$\mathbf{8}$\\
        3&4&4&30&$\infty$\\
        3&4&5&$\bm{\infty}$&$\bm{\infty}$\\
        3&4&6&$\mathbf{6}$&$\mathbf{6}$\\
        \hline
        3&5&1&$\mathbf{20}$&$\mathbf{20}$\\
        3&5&2&$\mathbf{20}$&$\mathbf{20}$\\
        3&5&3&$\mathbf{20}$&$\mathbf{20}$\\
        3&5&4&40&$\infty$\\
        3&5&5&32&$\infty$\\
        3&5&6&$\mathbf{10}$&$\mathbf{10}$\\
        \hline
        3&6&1&$\mathbf{24}$&$\mathbf{24}$\\
        3&6&2&$\mathbf{24}$&$\mathbf{24}$\\
        3&6&3&$\mathbf{24}$&$\mathbf{24}$\\
        3&6&4&$\mathbf{24}$&$\mathbf{24}$\\
        3&6&5&36&$\infty$\\
        3&6&6&$\mathbf{18}$&$\mathbf{18}$\\
        3&6&7&$\mathbf{18}$&$\mathbf{18}$\\
        3&6&8&36&$\infty$\\
        3&6&9&$\mathbf{16}$&$\mathbf{16}$\\
        3&6&10&36&$\infty$\\
        3&6&11&$\bm{\infty}$&$\bm{\infty}$\\
        3&6&12&$\mathbf{14}$&$\mathbf{14}$\\
        \hline
    \end{tabular}
    }
        \caption{$3 \leq g \leq 6$}
        \label{tab:table1}
    \end{subtable}
    \hspace{0.1\textwidth} 
    \begin{subtable}{0.40\textwidth}
        \centering
            \renewcommand{\arraystretch}{0.9}
            {\footnotesize
\begin{tabular}{|c| c | c | c | c |}
        \hline
        $k$ & $g$ & $\lambda$ & $n(k,g,\lambda) \geq$ & $n(k,g,\lambda) \leq$\\
        \hline
        3&7&1&42&56\\
        3&7&2&42&$\infty$\\
        3&7&3&42&56\\
        3&7&4&42&$\infty$\\
        3&7&5&42&$\infty$\\
        3&7&6&$\mathbf{28}$&$\mathbf{28}$\\
        3&7&7&$\mathbf{26}$&$\mathbf{26}$\\
        3&7&8&42&$\infty$\\
        3&7&9&42&$\infty$\\
        3&7&10&42&$\infty$\\
        3&7&11&$\bm{\infty}$&$\bm{\infty}$\\
        3&7&12&$\bm{\infty}$&$\bm{\infty}$\\
        \hline
        3&8&1&56&64\\
        3&8&2&52&64\\
        3&8&3&48&64\\
        3&8&4&48&50\\
        3&8&5&48&64\\
        3&8&6&$\mathbf{48}$&$\mathbf{48}$\\
        3&8&7&48&$\infty$\\
        3&8&8&$\mathbf{42}$&$\mathbf{42}$\\
        3&8&9&$\mathbf{48}$&$\mathbf{48}$\\
        3&8&10&$\mathbf{44}$&$\mathbf{44}$\\
        3&8&11&$\mathbf{40}$&$\mathbf{40}$\\
        3&8&12&$\mathbf{40}$&$\mathbf{40}$\\
        3&8&13&48&$\infty$\\
        3&8&14&48&$\infty$\\
        3&8&15&48&$\infty$\\
        3&8&16&48&$\infty$\\
        3&8&17&48&$\infty$\\
        3&8&18&48&$\infty$\\
        3&8&19&48&$\infty$\\
        3&8&20&48&$\infty$\\
        3&8&21&48&$\infty$\\
        3&8&22&48&$\infty$\\
        3&8&23&$\bm{\infty}$&$\bm{\infty}$\\
        3&8&24&$\mathbf{30}$&$\mathbf{30}$\\
        \hline
    \end{tabular}
    }
        \caption{$7 \leq g \leq 8$}
        \label{tab:3RegTable2}
    \end{subtable}
    \caption{An overview of the best lower and upper bounds for $n(3,g,\lambda)$.}
    \label{tab:comparison3Reg}
\end{table}

\begin{table}[htbp]
    \centering
    \begin{subtable}{0.40\textwidth}
        \centering
 \renewcommand{\arraystretch}{0.9}
    {\footnotesize
    \begin{tabular}{|c| c | c | c | c |}
        \hline
        $k$ & $g$ & $\lambda$ & $n(k,g,\lambda) \geq$ & $n(k,g,\lambda) \leq$\\
        \hline
        4&3&1&$\mathbf{9}$&$\mathbf{9}$\\
        4&3&2&$\mathbf{9}$&$\mathbf{9}$\\
        4&3&3&$\mathbf{7}$&$\mathbf{7}$\\
        4&3&4&$\mathbf{6}$&$\mathbf{6}$\\
        4&3&5&$\bm{\infty}$&$\bm{\infty}$\\
        4&3&6&$\mathbf{5}$&$\mathbf{5}$\\
        \hline
        4&4&1&$\mathbf{20}$&$\mathbf{20}$\\
        4&4&2&$\mathbf{18}$&$\mathbf{18}$\\
        4&4&3&$\mathbf{16}$&$\mathbf{16}$\\
        4&4&4&$\mathbf{13}$&$\mathbf{13}$\\
        4&4&5&$\mathbf{12}$&$\mathbf{12}$\\
        4&4&6&$\mathbf{14}$&$\mathbf{14}$\\
        4&4&7&$\mathbf{24}$&$\mathbf{24}$\\
        4&4&8&$\mathbf{11}$&$\mathbf{11}$\\
        4&4&9&$\mathbf{12}$&$\mathbf{12}$\\
        4&4&10&$\mathbf{10}$&$\mathbf{10}$\\
        4&4&11&24&$\infty$\\
        4&4&12&$\mathbf{10}$&$\mathbf{10}$\\
        4&4&13&24&$\infty$\\
        4&4&14&22&$\infty$\\
        4&4&15&24&$\infty$\\
        4&4&16&$\bm{\infty}$&$\bm{\infty}$\\
        4&4&17&$\bm{\infty}$&$\bm{\infty}$\\
        4&4&18&$\mathbf{8}$&$\mathbf{8}$\\
        \hline
        4&5&1&35&420\\
        4&5&2&$\mathbf{30}$&$\mathbf{30}$\\
        4&5&3&30&40\\
        4&5&4&$\mathbf{30}$&$\mathbf{30}$\\
        4&5&5&28&30\\
        4&5&6&30&55\\
        4&5&7&30&$\infty$\\
        4&5&8&$\mathbf{25}$&$\mathbf{25}$\\
        4&5&9&$\mathbf{25}$&$\mathbf{25}$\\
        4&5&10&$\mathbf{24}$&$\mathbf{24}$\\
        4&5&11&$\mathbf{30}$&$\mathbf{30}$\\
        4&5&12&$\mathbf{20}$&$\mathbf{20}$\\
        4&5&13&30&$\infty$\\
        4&5&14&30&$\infty$\\
        4&5&15&28&$\infty$\\
        4&5&16&30&$\infty$\\
        4&5&17&30&$\infty$\\
        4&5&18&$\bm{\infty}$&$\bm{\infty}$\\
        \hline
    \end{tabular}
    }
        \caption{$3 \leq g \leq 5$}
        \label{tab:4RegTable1}
    \end{subtable}
    \hspace{0.1\textwidth} 
    \begin{subtable}{0.40\textwidth}
        \centering
            \renewcommand{\arraystretch}{0.9}
            {\footnotesize
\begin{tabular}{|c| c | c | c | c |}
        \hline
        $k$ & $g$ & $\lambda$ & $n(k,g,\lambda) \geq$ & $n(k,g,\lambda) \leq$\\
        \hline
        4&6&1&54&1152\\
        4&6&2&51&84\\
        4&6&3&50&1152\\
        4&6&4&48&96\\
        4&6&5&48&$\infty$\\
        4&6&6&46&60\\
        4&6&7&48&648\\
        4&6&8&45&60\\
        4&6&9&44&512\\
        4&6&10&42&81\\
        4&6&11&42&$\infty$\\
        4&6&12&40&64\\
        4&6&13&42&$\infty$\\
        4&6&14&39&60\\
        4&6&15&38&$\infty$\\
        4&6&16&39&48\\
        4&6&17&42&$\infty$\\
        4&6&18&$\mathbf{35}$&$\mathbf{35}$\\
        4&6&19&36&$\infty$\\
        4&6&20&36&48\\
        4&6&21&36&44\\
        4&6&22&36&42\\
        4&6&23&36&42\\
        4&6&24&36&40\\
        4&6&25&36&42\\
        4&6&26&36&42\\
        4&6&27&36&38\\
        4&6&28&36&42\\
        4&6&29&36&$\infty$\\
        4&6&30&$\mathbf{35}$&$\mathbf{35}$\\
        4&6&31&$\mathbf{36}$&$\mathbf{36}$\\
        4&6&32&$\mathbf{36}$&$\mathbf{36}$\\
        4&6&33&$\mathbf{36}$&$\mathbf{36}$\\
        4&6&34&$\mathbf{36}$&$\mathbf{36}$\\
        4&6&35&36&$\infty$\\
        4&6&36&$\mathbf{32}$&$\mathbf{32}$\\
        4&6&37&36&$\infty$\\
        4&6&38&36&$\infty$\\
        4&6&39&$\mathbf{32}$&$\mathbf{32}$\\
        4&6&40&36&$\infty$\\
        4&6&41&36&$\infty$\\
        4&6&42&$\mathbf{30}$&$\mathbf{30}$\\
        4&6&43&$\mathbf{30}$&$\mathbf{30}$\\
        4&6&44&36&$\infty$\\
        4&6&45&36&$\infty$\\
        4&6&46&36&$\infty$\\
        4&6&47&36&$\infty$\\
        4&6&48&$\mathbf{28}$&$\mathbf{28}$\\
        4&6&49&36&$\infty$\\
        4&6&50&36&$\infty$\\
        4&6&51&36&$\infty$\\
        4&6&52&$\bm{\infty}$&$\bm{\infty}$\\
        4&6&53&$\bm{\infty}$&$\bm{\infty}$\\
        4&6&54&$\mathbf{26}$&$\mathbf{26}$\\
        \hline
    \end{tabular}
    }
        \caption{$g=6$}
        \label{tab:4RegTable2}
    \end{subtable}
    \caption{An overview of the best lower and upper bounds for $n(4,g,\lambda)$.}
    \label{tab:comparison4Reg}
\end{table}

\section{Conclusion}
Throughout our paper, we have repeatedly referred to the tables below. They can be simultaneously viewed as a list of best results obtained in the area so far, but also as a source of ideas and inspiration.

For example, consulting the information concerning $ n(3,7,3) $, we see a gap between the lower bound of $42$ 
on the order of any $vgr(n,3,7,3)$-graph and the order of a smallest
$ vgr(n,3,7,3) $-graph found to day; equal to $56$. While we do not know whether a smaller $ vgr(n,3,7,3) $-graph exists, we note that the $ vgr(56,3,7,3) $-graph is a celebrated graph in topological graph theory. It is the underlying graph of the {\em Klein map}, a regular polyhedron of type $(7,3)$ and genus $3$, whose polyhedral representation was described by Schulte and Wills in 1985 \cite{SW85}. The relevance of the polyhedral representation lies in the fact
that locally each vertex $v$ of the map is adjacent to three faces of the polyhedron, all of which are of length $7$ and their borders
constitute the only $7$-cycles passing through $v$. Since $7$ is also the girth of the underlying graph of the Klein map, the graph represents a specific example of a connection between $vgr(n,k,g,k)$-graphs and maps of type
$(g,k)$. This connection certainly deserves further investigation.

\section*{Acknowledgements}
\noindent The computational resources and services used in this work were provided by the VSC (Flemish Supercomputer Centre), funded by the Research Foundation Flanders (FWO) and the Flemish Government - Department EWI. Robert Jajcay is partially supported by VEGA 1/0437/23 and SK-AT-23-0019. Jorik Jooken is supported by a Postdoctoral Fellowship of the Research Foundation Flanders (FWO) with grant number 1222524N. The research of István Porupsánszki was supported in part by the Hungarian National Research, Development and Innovation Office  OTKA grant no. SNN 132625. The authors are also grateful to Jan Goedgebeur for helping to make the graphs publicly available on the House of Graphs.
	

\begin{thebibliography}{10}
        \bibitem{APKP23}
        G. Araujo-Pardo, Gy. Kiss, I. Porupsánszki, On extremal (almost) edge-girth-regular
graphs. arXiv:2401.15411 [math.CO], 2023.
       \bibitem{CD02}
       M.~Conder and P.~Dobcs\'anyi.
       \newblock Trivalent symmetric graphs on up to 768 vertices.
       \newblock {\em J. Comb. Math. \& Comb. Comp.}, 40:41--63, 2002.

        \bibitem{CGJ23}
        S.~Cambie, J.~Goedgebeur, and J.~Jooken.
        \newblock The maximum number of connected sets in regular graphs.
        \newblock arXiv:2311.00075 [math.CO], 2023.
        
        \bibitem{CDG23}
       K.~Coolsaet, S.~D’hondt, and J.~Goedgebeur.
       \newblock House of Graphs 2.0: A
database of interesting graphs and more.
       \newblock {\em Discrete Appl. Math.}, 325:97--107, 2023. Available at \url{https://houseofgraphs.org/}.


       \bibitem{DFJR21}
       A.~Z.~Drglin, S.~Filipovski, R.~Jajcay and T.~Raiman.
       \newblock Extremal edge-girth-regular graphs.
       \newblock {\em Graphs and Combin.}, 37:2139--2154, 2021.
    \bibitem{D24}
    L.~Droogendijk. \newblock
    Nonexistence of certain edge-girth-regular graphs. 
     \newblock arXiv:2403.20049 [math.CO], 2024.
     
        \bibitem{EJ08}
       G.~Exoo and R.~Jajcay.
       \newblock Dynamic cage survey.
       \newblock {\em Electron. J. Combin.}, DS16:48, 2008.

        \bibitem{GJ24}
        J.~Goedgebeur and J.~Jooken.
        \newblock{Exhaustive generation of edge-girth-regular graphs.}
		\newblock arXiv:2401.08271 [math.CO], 2024.

        \bibitem{HR20}
       D.~Holt and G.~Royle.
       \newblock A census of small transitive groups and vertex-transitive graphs.
       \newblock {\em J. Symbolic Comput.}, 101:51--60, 2020.
       
       \bibitem{JKM18}
       R.~Jajcay, G.~Kiss, and \v{S}.~Miklavi\v{c}.
       \newblock Edge-girth-regular graphs.
       \newblock {\em European J. Combin.}, 72:70--82, 2018.

       \bibitem{KMS22}
       Gy. Kiss, \v S. Miklavi\v c, and T. Sz\H{o}nyi.
       \newblock A stability result for girth-regular graphs with even girth. 
       \newblock {\em J. Graph Theory}, 100 (1): 163--181, 2022.

        \bibitem{Me99}
       M. Meringer.
       \newblock Fast generation of regular graphs and construction of cages.
       \newblock {\em J. Graph Theory}, 30(2):137--146, 1999.
        \bibitem{P23}
I. Porupsánszki, \emph{On edge-girth-regular graphs: lower bounds and new families}. arXiv:2305.17014 [math.CO], 2023.
        \bibitem{PSV13}
       P.~Poto\v cnik, P. Spiga and G. Verret.
       \newblock Cubic vertex-transitive graphs on up
to 1280 vertices.
       \newblock {\em J. Symbolic Comput.}, 50:465--477, 2013.

        \bibitem{PSV15}
       P.~Poto\v cnik, P. Spiga and G. Verret.
       \newblock Bounding the order of the vertex-stabiliser in 3-valent vertex transitive and 4-valent arc-transitive graphs.
       \newblock {\em J. Combin. Theory Ser. B}, 111:148--180, 2015.
        
       \bibitem{PV19}
       P.~Poto\v cnik and J. Vidali.
       \newblock Girth-regular graphs.
       \newblock {\em Ars Math. Contemp.}, 17 (2): 349--368, 2019.

       
       \bibitem{S63}
       H.~Sachs.
       \newblock Regular graphs with given girth and restricted circuits. 
       \newblock {\em J. London Math. Soc.}, 38:423--429, 1963.

       \bibitem{SW85}
       E.~Schulte and J.~M.~Wills.
    \newblock A polyhedral realization of Felix Klein’s map $\{3,7\}_8$ on a Riemann surface of genus $3$.
    \newblock {\em J. Lond. Math. Soc.}, II. 32:539--547, 1985.
	\end{thebibliography}
\end{document}